\documentclass[reqno, 10pt]{amsart}
\usepackage{color}
\usepackage{amsmath}
\usepackage{amsfonts}
\usepackage{amssymb}
\usepackage{amsthm}
\usepackage{newlfont}
\usepackage{graphicx}

\newtheorem{thm}{Theorem}[section]

\newtheorem{lem}[thm]{Lemma}
\newtheorem{prop}[thm]{Proposition}
\theoremstyle{definition}

\theoremstyle{remark}
\newtheorem{rem}[thm]{Remark}

\numberwithin{equation}{section}
\numberwithin{thm}{section}

\newcommand{\eps}{\varepsilon}

\newcommand{\lsm}{\lesssim}

\newcommand{\R}{{\mathbb{R}}}
\renewcommand{\v}[1]{\ensuremath{\mathbf{#1}}}

\newcommand{\vv}{\v v}
\newcommand{\jnab}{\langle \nabla \rangle}
\newcommand{\ed}{\end {document}}

\newcounter{smalllist}

\newcommand{\jt}{\langle t \rangle}
\newcommand{\lrx}{\langle |\xi|\rangle}
\newcommand{\lre}{\langle |\eta|\rangle}
\newcommand{\nai}{{\frac{\nabla_\eta a_i(\xi,\eta)}{|\nabla_\eta a_i(\xi,\eta)|^2}}}
\newcommand{\gai}{{|\nabla_\eta a_i(\xi,\eta)|}}
\newcommand{\ai}{{a_i(\xi,\eta)}}

\title[global solution for Euler Poisson]{Smooth global solutions for the two dimensional
Euler Poisson system}
\author[J. Jang]{Juhi Jang}
\address{University of California
Riverside, CA 92521} \email{juhi.jang@ucr.edu}
\author[D. Li]{Dong Li}
\address{Department of Mathematics, University of Iowa, 14 MacLean Hall, Iowa City, IA 52242}%
\email{mpdongli@gmail.com}
\author[X. Zhang]{Xiaoyi Zhang}
\address{Department of Mathematics, University of Iowa, 14 MacLean Hall, Iowa City, IA 52242}%
\email{zh.xiaoyi@gmail.com}

\begin{document}
\maketitle
\begin{abstract}
The Euler-Poisson system is a fundamental two-fluid model to describe the dynamics
of the plasma consisting of compressible electrons and a uniform ion background. By
using the dispersive Klein-Gordon effect, Guo \cite{Guo98} first constructed a global
smooth irrotational solution in the three dimensional case. It has been conjectured that
same results should hold in the two-dimensional case. The main difficulty in 2D comes from
 the slow dispersion of the linear flow and certain nonlocal resonant obstructions in the nonlinearity.
 In this paper we develop a new method to overcome these difficulties and
 construct smooth global solutions for the 2D Euler-Poisson system.
\end{abstract}

\section{Introduction}
The Euler-Poisson system is a fundamental two-fluid model used to describe the
dynamics of a plasma consisting of moving electrons and ions. Since the
heavy ions move much more slowly than the light electrons, one can treat
them as two independent compressible fluids which only interact through their
self-consistent electromagnetic fields. In the simplest approximation, the ions
are immobile and uniformly distributed in space, providing only a background of
positive charge for the electrons. Neglecting magnetic effects, the equations
of motion describing the dynamics of the compressible electron fluid is then given
by the following Euler-Poisson system in $(t,x) \in [0,\infty) \times \mathbb R^d$,
\begin{align} \label{eq_EP_1}
\begin{cases}
\partial_t n + \nabla \cdot (n \v u) =0, \\
m_e n (\partial_t \v u + (\v u\cdot \nabla) \v u ) + \nabla p(n) = e n \nabla \phi, \\
\Delta \phi = 4\pi e (n-n_0).
\end{cases}
\end{align}

Here $n=n(t,x)$ and $\v u=\v u(t,x)$ denote the density and average velocities of the electrons respectively.
The dimension $d$ will mostly be either $d=3$ or $d=2$ and we shall state clearly the dimensional dependence
of the results. The state of the plasma is completely specified once the density and average velocities are known.
It is useful to keep in mind that such a fluid model description is only accurate when the plasma velocity distribution is
close to the Maxwell-Boltzmann distribution so that the microscopic details of each individual particles
are averaged out. The symbol $e$ and $m_e$ denote the unit charge and
mass of electrons. The first equation in
\eqref{eq_EP_1} is simply the mass conservation law, while the second equation comes from the momentum
balance. The pressure term $p(n)$ is assumed to obey the polytropic $\gamma$-law, i.e.
\begin{align} \label{eq_EP_2}
p(n) = A n^\gamma,
\end{align}
where $A$ is a constant and $\gamma \ge 1$ is usually called the adiabatic index.
The last equation in \eqref{eq_EP_1} is the Poisson equation (or
Gauss law) which computes the electric field self-consistently through the charge distribution.
We assume here that at the equilibrium the density of ions and electrons are both a constant denoted by $n_0$.
The dependence on ions only enter through the term $4\pi e (n-n_0)$ which is the net charge generating
the electric field. It is in this sense that the system \eqref{eq_EP_1} is the simplest two-fluid model
for plasmas since the ions are merely treated as constant uniform sources.
The natural boundary condition for the electric potential is a decaying condition at infinity,
i.e.
\begin{align} \label{eq_EP_3}
\lim_{|x| \to \infty} \phi(t,x) = 0.
\end{align}
Throughout the rest of this paper, we shall consider an irrotational flow\footnote{In the two dimensional
case, one can regard in the usual fashion $\v u=(u_1,u_2,0)$, for which the
condition $\nabla \times \v u \equiv 0$ simplifies
to $\partial_{x_1} u_2 - \partial_{x_2} u_1$.}, i.e.
\begin{align} \label{eq_EP_4}
\nabla \times \v u \equiv 0.
\end{align}
By a simple computation, it is easy to check that the irrotational condition is preserved for all
time.
The system \eqref{eq_EP_1}--\eqref{eq_EP_4} describes the dynamics of an irrotational
compressible Eulerian electronic fluid
moving against the ionic forces.

Our main objective is to construct global in time smooth solutions for the Euler-Poisson system.
This is not an easy task since this is a hyberbolic conservation law with zero dissipation which
has remain largely unsolved and no general theory is available. The Euler-Poisson system has
close connection with the well-known compressible Euler equations. Indeed in \eqref{eq_EP_1}
if the electric field term $\nabla \phi$ is dropped, one recovers the usual Euler equations
for compressible fluids. In \cite{Si85}, Sideris considered the 3D compressible Euler equation
for a classical polytropic ideal gas with adiabatic index $\gamma>1$. For a class of
initial data which coincide with a constant state outside a ball, he proved that the lifespan
of the corresponding $C^1$ solution must be finite. In \cite{Ra89} Rammaha extended this
result to the 2D case. For the Euler-Poisson system, Guo and Tahvildar-Zadeh \cite{GuoTZ99}
established a "Siderian" blowup result for spherically symmetric initial data.
Recently Chae and Tadmor \cite{CT08} proved finite-time blow-up for $C^1$ solutions of
a class of pressureless attractive Euler-Poisson equations in $\mathbb R^n$, $n\ge 1$.
These negative results showed the abundance of shock waves for large solutions.

The point of departure in the analysis of \eqref{eq_EP_1} from that of the compressible Euler
equations begins with understanding the small irrotational perturbations of the equilibrium state
$n\equiv n_0$, $\v u\equiv 0$. For the 3D compressible Euler equation with irrotational initial data
$(n,\v u) = (\epsilon \rho + n_0, \epsilon \v v)$, where $\rho \in \mathcal S(\mathbb R^3)$,
$\v v\in \mathcal S(\mathbb R^3)^3$ ($\mathcal S(\mathbb R^3)$ is the usual Schwartz space), Sideris
\cite{Si91} proved that the lifespan of the classical solution $T_\epsilon > \exp( C/\epsilon)$.
For the upper bound it follows from his previous paper \cite{Si85} that $T_\epsilon < \exp(C/\epsilon^2)$
under some mild conditions on the initial data. Sharper results are also available. For initial
data which is spherically symmetric and is smooth compact $\epsilon$-perturbation of the constant state,
Godin \cite{Godin05} obtained by using a suitable approximation solution the precise asymptotic of
the lifespan $T_\epsilon$ as
\begin{align*}
\lim_{\epsilon \to 0} \epsilon \log T_\epsilon = T^*,
\end{align*}
where $T^*$ is a constant. These results rely crucially on the observation that after some simple manipulations,
the compressible Euler equation in rescaled variables is given by a vectorial nonlinear wave equation with pure
quadratic nonlinearities. The linear part of the wave equation decays at most at the speed $t^{-(d-1)/2}$ which
in 3D case is not integrable. Unless some additional structure e.g. the null condition \cite{C86, K86}
is imposed on the nonlinearity, one cannot in general expect global existence of small solutions.
For the Euler-Poisson system \eqref{eq_EP_1}, the situation for small irrotational perturbations
of the equilibrium state is quite different. By a straightforward computation, the linear part of the
Euler-Poisson system for irrotational flows has the form
\begin{align*}
\begin{cases}
\partial_{tt} n - \frac 1 {m_e} p^\prime (n_0) \Delta n + \omega_p^2 n = 0, \\
\partial_{tt}\v u - \frac 1 {m_e} p^\prime (n_0) \Delta \v u + \omega_p^2 \v u=0.
\end{cases}
\end{align*}
Here $\omega_p = \sqrt{ 4\pi e^2 n_0 /m_e }$ is so called plasma frequency. This new term $\omega_p^2$
which is absent in the pure Euler case makes the linear part of the Euler-Poisson system into a Klein-Gordon
system for which the linear solutions have an enhanced decay of $(1+t)^{-d/2}$. This is in sharp contrast with
the pure Euler case for which the decay is only $t^{-(d-1)/2}$. Note that in $d=3$, $(1+t)^{-d/2}=(1+t)^{-3/2}$
which is integrable in $t$. In a remarkable paper \cite{Guo98}, by exploiting the crucial decay property of the
Klein-Gordon flow in 3D, Guo \cite{Guo98} modified Shatah's normal form method \cite{Sh85} and
constructed a smooth irrotational global solution to \eqref{eq_EP_1} around the
equilibrium state $(n_0,0)$ for which the perturbations decay at a rate $C_p \cdot (1+t)^{-p}$ for any
$1<p<3/2$ (here $C_p$ denotes a constant depending on the parameter $p$). Note in particular that
the sharp decay $t^{-3/2}$ is marginally missed here due to a technical complication, namely in employing
one of the $L^p \to L^\infty$ decay estimates, the nonlinear term has a Riesz-type singular operator in the
front and one is forced to take $p>1$ in bounding the corresponding part.

As was already mentioned, the purpose of this paper is to construct smooth global solutions to
\eqref{eq_EP_1} for the two-dimensional case. One can already sense a bit that this situation
is more challenging than the $d=3$ case since the linear solutions to the Klein-Gordon system in
$d=2$ decays only at $(1+t)^{-1}$ which is not integrable, in particular making the
strategy in \cite{Guo98} difficult to apply. As we shall see shortly, the Euler-Poisson
system after writing in terms of the perturbed variables can be recast into a quasi-linear Klein-Gordon
system with nonlocal nonlinearities.  To understand and spell out the main difficulties
associated with the analysis of this nonlocal quasi-linear Klein-Gordon system, we first review some
known results on nonlinear Klein-Gordon equations. For general scalar quasi-linear Klein-Gordon equations
in 3D with quadratic type nonlinearities, global small smooth solutions were constructed independently by
Klainerman \cite{K85} using the invariant vector field method and Shatah \cite{Sh85}
using a normal form method. There are
essential technical difficulties to employ Klainerman's invariant vector field method to the Klein-Gordon
system associated with \eqref{eq_EP_1}. Namely after reformulating in terms of perturbed the variables,
the nonlinear part of the system \eqref{eq_EP_1} has a Riesz type nonlocal term which comes from solving the
Poisson equation for the electric field in \eqref{eq_EP_1}, i.e. a term of the form
\begin{align*}
(-\Delta)^{-1} \partial_{x_i} \partial_{x_j} \bigl( \text{quadratic terms} \bigr).
\end{align*}
The Klainerman invariant vector fields consist of infinitesimal generators which commute well with
the linear operator $\partial_{tt} -\Delta +1$. The most problematic part comes from the Lorentz boost
$\Omega_{0j} = t \partial_{x_j} + x_j \partial_t$. Whilst the first part $t \partial_{x_j}$ commutes
naturally with the Riesz operator $R_{ij}=(-\Delta)^{-1} \partial_{x_i} \partial_{x_j}$,
the second part $x_j \partial_t $ interacts rather badly with $R_{ij}$ and produces a commutator
which scales as
\begin{align*}
[x_j \partial_t, R_{ij}] \sim \partial_t |\nabla|^{-1}.
\end{align*}
After repeated commutation of these operators one obtain in general terms of the form
$|\nabla|^{-N}$ which makes the low frequency part of the solution out of control.
It is for this reason that in 3D case Guo \cite{Guo98} adopted Shatah's method of normal
form in $L^p$ ($p>1$) setting for which the Riesz term $R_{ij}$ causes no trouble.
We turn now to the 2D Klein-Gordon equations with pure quadratic nonlinearities. In this case, direct
applications of either Klainerman's invariant vector field method or Shatah's normal form method
are not possible since the linear solutions only decay at a speed of $(1+t)^{-1}$ which is not integrable
and makes the quadratic nonlinearity quite resonant. In \cite{ST93}, Simon and Taflin constructed wave
operators for the 2D semilinear Klein-Gordon system with quadratic nonlinearities.
In \cite{Oz96}, Ozawa, Tsutaya and Tsutsumi considered the Cauchy problem and
constructed smooth global solutions by first transforming the quadratic nonlinearity into a cubic one
using Shatah's normal form method and then applying Klainerman's invariant vector field method to obtain
decay of intermediate norms. Due to the nonlocal complication with the Lorentz boost
which we explained earlier, this approach
seems difficult to apply in the 2D Euler-Poisson system.

Our approach in this paper is inspired by the recent work of Gustafson, Nakanishi and Tsai \cite{GNT06}
on the Gross-Pitaevskii equation of the form
\begin{align*}
i\partial_t \psi = -\Delta \psi + ( |\psi|^2-1) \psi,
\end{align*}
where $\psi: \; \mathbb R^{1+d} \to \mathbb C$  is solved with the boundary condition
\begin{align*}
|\psi(t,x) | \to 1, \quad \text{as $|x| \to \infty$}.
\end{align*}
The main objective in \cite{GNT06} is to investigate large time behavior of solutions
$\psi=1+$"small". It is then natural to look at the perturbation $u=\psi-1$ which satisfies
the equation
\begin{align} \label{972010_2}
i\partial_t u + \Delta u -2 Re(u) = F(u), \quad F(u):= u^2 +2|u|^2 +|u|^2 u.
\end{align}
To make the problem complex linear, introduce change of variable
\begin{align*}
v= \mathcal Ku := \sqrt{ (-\Delta)^{-1} (2-\Delta) } Re(u) +i Im(u).
\end{align*}
Then for $v$ the equation takes the form
\begin{align} \label{972010_1}
i \partial_t v - \sqrt{(-\Delta) (2-\Delta) } v = - i \mathcal K i F(\mathcal K^{-1} v).
\end{align}
In \cite{GNT06}, Gustafson, Nakanishi and Tsai considered the final data problem for
\eqref{972010_1} (i.e. Cauchy data at $t=\infty$),
\begin{align*}
v(t) = e^{- i \sqrt{-\Delta (2-\Delta)} t } \phi
+ \int_t^\infty e^{- i \sqrt{-\Delta (2-\Delta) } (t-s) }
\mathcal K i F(\mathcal K^{-1} v(s) ) ds,
\end{align*}
where $\phi$ is the final data. Under some decay and smallness assumptions on $\phi$,
they constructed a unique global solution for \eqref{972010_1} in dimensions
$d=2,3$. In particular for $d=2$, the solutions has a decay
\begin{align} \label{972010_3}
\| v(t) -
e^{ - i \sqrt{-\Delta(2-\Delta)} t } \phi \|_{\dot H^1}
\le C_{\epsilon} t^{ -1+\epsilon },
\end{align}
for any $\epsilon \in (0,1)$ and $C_{\epsilon}$ is a constant depending on $\epsilon$.
Note that this marginally misses the sharp decay $t^{-1}$ due to some technical issues
which we explain now. The first step in \cite{GNT06} is to estimate the first nonlinear
iterate, i.e.
\begin{align*}
z_1(t) = \int_t^\infty e^{- i \sqrt{-\Delta (2-\Delta) } (t-s) }
\mathcal K i F(\mathcal K^{-1} v_1(s) ) ds,
\end{align*}
where
$$ v_1(t) =e^{ - i \sqrt{-\Delta(2-\Delta)} t } \phi $$
is simply the linear solution. By a careful space-time  phase estimate (see Lemma 4.1 in \cite{GNT06}),
the authors showed that
\begin{align*}
\| z_1(t) \|_{\dot H^1} \lesssim t^{-1} (\log t)^2, \quad t\ge 2,
\end{align*}
here we suppressed the dependence on the final data $\phi$. The logarithm part of $t$  comes mainly
from certainly degeneracies of the phase near $\xi =0$. Since the nonlinearity is quadratic,
one immediately observe that due to this logarithm loss we cannot close the simple $\dot H^1$ estimates for further
iterates (in fact there also some problems with low frequencies which we shall not dwell further here).
To remedy this problem, Gusatfson, Nakanishi and Tsai then introduced a novel normal form
for the original equation written in the function $u$ (see \eqref{972010_2})
\begin{align*}
w = u + P_{\le 1} ( |u|^2/2),
\end{align*}
where $P_{\le 1}$ is the usual Littlewood-Paley projector localized to frequency $|\xi| \lesssim 1$.
After this transformation and using Strichartz estimates together with a fixed point argument,
the authors were able to settle for this logarithmic loss and obtain the existence of global solutions,
in particular establishing \eqref{972010_3}.

Our analysis of the Euler-Poisson system starts with sharpening the
analysis of the Gross-Pitaevskii equation in 2D. In \cite{Liprep},
the second author of this paper observed that provided with some
suitable transform and choosing a good working space, one can remove
the epsilon in \eqref{972010_3}, in particular obtaining the sharp
decay $1/t$ in 2D. A key observation in this work is that one should
obtain the critical $1/t$ decay in the energy estimates of the first
nonlinear estimate, and after that only energy estimates are needed
(in particular, no Strichartz estimates are needed!). This intuition
is further arrested in our subsequent work \cite{LZprep}, where we
revisited the classical semilinear Klein-Gordon equations in 2D with
quadratic nonlinearity and also semilinear wave equations in 3D with
a quadratic nonlinearity satisfying the null condition. Note that in
both cases the linear solution decay at the critical speed $1/t$ for
which the quadratic nonlinearity is resonant. In that work we
constructed wave operators for the final data problem and obtain the
sharp decay of the nonlinear part of the solution. As a byproduct of
our analysis in \cite{LZprep}, one can reinforce the notion of the
null condition as an annihilation condition on the quadratic
interaction of linear solutions.

In this work, we take the point of view that the Euler-Poisson system is
a quasi-linear Klein-Gordon system with nonlocal quadratic nonlinearities.
We shall construct smooth global solutions to this system by solving
the final data problem (i.e constructing the wave operators). As is well-known
with quasilinear systems, the Strichartz estimates will suffer loss of derivatives unless
one considers energy estimates. We shall resolve this difficulty by performing a delicate localization
in time argument. This localization is only carried out for the part of the solution which
has no resonances and only for the corresponding quasi-linear part of the nonlinearity. We then introduce
a new iteration scheme to construct the solution as a uniform limit of iterates. The crucial
point in the iteration step is to obtain a priori uniform sharp decay estimates of $H^m$ norms of
the nonresonant part of the solution. This is done by a delicate bootstrap argument. We should
stress that our new iteration scheme
is quite robust and should have applications to many other systems for which the nonlinearity is resonant.
To better illustrate our technique, we summarize below the main steps of the proof.

\texttt{Step 1}. Reformulation. For simplicity set all physical constants $e$, $m_e$, $4\pi$ and $A$ to be
one. To simplify the presentation, we also set $\gamma=3$. Define the rescaled functions
\begin{align*}
u(t,x) & = \frac{n(t/c_0, x) -n_0} {n_0}, \\
\v v(t,x) & = \frac 1 {c_0} \v u (t/c_0, x),
\end{align*}
where the sound speed is $c_0 = \sqrt 3 n_0$. For simplicity we shall set $n_0 = 1/3$ so that the characteristic
wave speed is unity. Then in terms of the new functions $(u, \v v)$, the Euler-Poisson system
\eqref{eq_EP_1}--\eqref{eq_EP_4} takes the form
\begin{align*}
\begin{cases}
(\square+1) u = \nabla \cdot \Bigl( (\vv \cdot \nabla) \vv + u
\nabla u \Bigr) - \partial_t \Bigl( \vv \cdot \nabla u + u \nabla \cdot \vv \Bigr) \\
(\square+1) \vv = \nabla\Bigl( (\vv \cdot\nabla )u + u \nabla \cdot \vv \Bigr)
-\partial_t \Bigl( (\vv \cdot \nabla) \vv + u \nabla u \Bigr) - \nabla \Delta^{-1} \nabla \cdot (u \vv),
\end{cases}
\end{align*}
or in a more compact notation
\begin{align*}
\begin{cases}
(\square +1) u = G_1(u,\vv) \\
(\square+1) \vv = G_2(u,\vv).
\end{cases}
\end{align*}

\texttt{Step 2}. Remove resonances and fine decomposition of the solution. Define linear
solutions according to \eqref{eha}--\eqref{ehb}. We then identify the resonant part of the solution
corresponding to the quadratic interaction of the linear solutions. This is done by solving
the system (see \eqref{ePhia}--\eqref{ePhib})
\begin{align*}
\begin{cases}
(\square+1) \Phi_1 & =G_1(h_1,h_2), \\
(\square+1) \Phi_2 & = G_2(h_1,h_2)
\end{cases}
\end{align*}
with a decaying condition at $t=\infty$. We then decompose the solution as
\begin{align*}
\begin{cases}
u= \bar u + h_1 + \Phi_1, \\
\v v = \bar{\v v}+ h_2 + \Phi_2.
\end{cases}
\end{align*}
The new variable $(\bar u, \bar{\v v})$ no longer contains resonances and satisfy
\begin{align*}
\begin{cases}
(\square +1) \bar u &= \tilde G_1(\bar u, \bar{\v v}, h, \Phi), \\
(\square +1) \bar{\v v} &= \tilde G_2(\bar u, \bar{\v v}, h, \Phi),
\end{cases}
\end{align*}
where $\tilde G_1$, $\tilde G_2$ are corresponding new nonlinearities.

\texttt{Step 3}. Sharp decay estimates for $h_1$, $h_2$ and the resonances $\Phi_1$, $\Phi_2$. In this
step we perform a careful space-time analysis of the bilinear integral for $\Phi_1$, $\Phi_2$.
Note that the $L^\infty$ norms of the source terms $h_1$, $h_2$ decay as $1/t$ which is a standard
$L^1-L^\infty$ estimate. The crucial point in our analysis is to establish sharp $1/t$ decay of $H^m$
norms of $\Phi_1$, $\Phi_2$. This is needed later for the iteration step.

\texttt{Step 4}. Localization in time and definition of the iteration scheme. Let $\chi \in C_c^\infty(\mathbb R)$
be such that $\chi(t)=1$ for $0\le t \le 1$ and $\chi(t) = 0$ for $t\ge 2$. For integers $n\ge 0$, define
\begin{align*}
\chi_n(t) = \chi(2^{-n} t).
\end{align*}
Observe that $\chi_n$ is localized to the regime $t \lesssim 2^n$. Let $\bar u^{(0)} =0$, $\bar{\v v}^{(0)} =0$.
For $n\ge 1$ inductively define

\begin{align*}
\begin{cases}
(\square+1) \bar u^{(n)} &= \chi_n (t) (\tilde G_{11}( \bar u^{(n)}, \cdot) + \tilde G_{12} (\bar{\v v}^{(n)}, \cdot))
+ \tilde G_{13} (\bar u^{(n-1)}, \bar{\v v}^{(n-1)}, h, \Phi), \\
(\square +1) \bar{\v v}^{(n)} & = \chi_n (t) ( \tilde G_{21}( \bar u^{(n)}, \cdot)
+ \tilde G_{22}( \bar{\v v}^{(n)}, \cdot) ) + \tilde G_{23} (\bar u^{(n-1)}, \bar {\v v}^{(n-1)}, h,\Phi).
\end{cases}
\end{align*}

Here $\tilde G_{11}$, $\tilde G_{12}$ denote the quasi-linear part of the nonlinearity in $\tilde G_1$. Similar convention
applies to $\tilde G_2$ (see \eqref{eq10a}--\eqref{eq10b} for more details). Note that the time localization is only
applied in the quasi-linear part of the nonlinearity containing only the interactions of the non-resonant part
of the solution.

\texttt{Step 5}. Solvability of the iteration system and non-uniform energy estimates.
The solvability of the iteration system is not immediately obvious, due to the quasi-linear
nature of the problem. In this respect our strategy is to take advantage of the smooth time cutoff
function $\chi_n(t)$ which vanishes for $t\ge 2^{n+1}$. In the large time regime $t\ge 2^{n+1}$,
the nonlinearity of the iteration system does not contain $(\bar u^{(n)}, \bar{\v v}^{(n)})$ and
we can solve for all $t\ge 2^{n+1}$ by using the fundamental solution with source term
depending only on the previous iteration $(\bar u^{(n-1)}, \bar{\v v}^{(n-1)} )$.
After this is done we solve
\textit{backwards} the same system for $t\le 2^{n+1}$ with initial data
$(\bar u^{(n)} (2^n), \bar{\v v}^{(n)} (2^n))$. We then perform an energy estimate. Denote
\begin{align*}
A_k(t) &= \| \bar u^{(k)} (t) \|_{H^m} + \| (\bar u^{(k)} (t) )^\prime \|_{H^m}
+ \| \bar \vv^{(k)}(t) \|_{H^m}  \\
 &\qquad+ \| ( \bar \vv^{(k)} (t) )^\prime \|_{H^m}, \quad \forall\, k\ge 0.
\end{align*}
Our inductive hypothesis is that
\begin{align*}
A_{n-1}(t) \le \frac{\epsilon} {\langle t \rangle}, \quad\forall\, t\ge 0.
\end{align*}
Here $\langle t \rangle = \sqrt{1+t^2}$ is the usual Japanese bracket notation.
We then use a crude energy estimate to obtain
\begin{align*}
A_n (t) \le \frac{ K_n \epsilon} {\langle t \rangle},
\end{align*}
where $K_n$ is a constant depending on the iteration $n$. This dependence on the iteration number $n$ comes from
the part where we solve the iteration system \textit{backwards} for $t\le 2^{n+1}$.

\texttt{Step 6}. Bootstrap and contraction. In this step we upgrade the estimate on $A_n$ and remove the dependence
on $n$. This is done by an infinite-time Gronwall argument (see Lemma \ref{gron_1}) which rests on the intuition
that nonlinear decay is dictated by the sharp decay of the source terms. After obtaining the uniform energy estimate and
closing the induction hypothesis, we perform the usual contraction argument. There is one more small twist in the
argument, namely we have to first show the strong contraction with a weak time decay $\langle t \rangle^{-1/2}$,
but this is enough to extract the limiting solution since the true decay $\langle t \rangle^{-1}$
from the iterates can be passed to
the limit for any finite $t$.

\noindent
\textbf{Organization of the paper}. In section \ref{sec_ref} we give the reformulation of
the 2D Euler-Poisson system and state the main results. In section \ref{sec_red} we extract
the non-resonant part of the solution and introduce the iteration scheme. Section \ref{sec_phi}
is devoted to the estimates of the resonant part of the solution. In section \ref{sec_solv}
we prove solvability of the iteration system and establish some non-uniform energy decay estimates.
In section \ref{sec_uni} we upgrade the decay estimate of the iterates and establish uniform in time
energy estimates. In the last section \ref{sec_contraction}, we perform a contraction argument and
prove the existence and uniqueness of the desired classical solution.

\section*{Acknowledgements}
The second author and the third author would like to thank the
Institute for Advanced Study for its hospitality, where the main
part of the work was conducted. J. Jang is supported in part by NSF
Grant DMS-0908007. D. Li is supported in part by NSF Grant-0908032.
X. Zhang is supported by an Alfred P. Sloan fellowship and also a
start-up funding from University of Iowa.

\section{Reformulation and main results} \label{sec_ref}
We first transform the system \eqref{eq_EP_1}--\eqref{eq_EP_4} in terms of certain perturbed
variables. For simplicity set all physical constants $e$, $m_e$, $4\pi$ and $A$ to be
one. To simplify the presentation, we also set $\gamma=3$ although other cases of $\gamma$ can be easily
included in the discussion. Define the rescaled functions
\begin{align*}
u(t,x) & = \frac{n(t/c_0, x) -n_0} {n_0}, \\
\v v(t,x) & = \frac 1 {c_0} \v u (t/c_0, x), \\
\psi(t,x) & = \phi (t/c_0, x),
\end{align*}
where the sound speed is $c_0 = \sqrt 3 n_0$. For simplicity we shall set $n_0 = 1/3$ so that the characteristic
wave speed is unity. The Euler-poisson system \eqref{eq_EP_1}--\eqref{eq_EP_4} in new variables take the form
\begin{align*}
\begin{cases}
\partial_t u + \nabla \cdot \v v + \v v \cdot \nabla u + u \nabla \cdot \v v =0, \\
\partial_t \v v + \nabla u + (\v v \cdot \nabla) \v v + u \nabla u = 3 \nabla \psi, \\
\partial_t \nabla \psi = -\frac 13 \v v - \frac 13 \nabla \Delta^{-1} \nabla \cdot( u \v v),\\
\Delta \psi = \frac 13 u.
\end{cases}
\end{align*}

Since the flow is irrotational $\nabla \times \v v \equiv 0$, we can then eliminate the
electric field $\nabla \psi$ using the Poisson equation. Take one more derivative and after some
simple computation, we arrive at
the klein-Gordon form of the 2D Euler-Poisson system written for the perturbed variables
$(u,\v v)$:
\begin{align} \label{e1}
\begin{cases}
(\square+1) u = \nabla \cdot \Bigl( (\vv \cdot \nabla) \vv + u
\nabla u \Bigr) - \partial_t \Bigl( \vv \cdot \nabla u + u \nabla \cdot \vv \Bigr) \\
(\square+1) \vv = \nabla\Bigl( (\vv \cdot\nabla )u + u \nabla \cdot \vv \Bigr)
-\partial_t \Bigl( (\vv \cdot \nabla) \vv + u \nabla u \Bigr) - \nabla \Delta^{-1} \nabla \cdot (u \vv).
\end{cases}
\end{align}
To simplify the discussion later, it is useful to denote the system as
\begin{align} \label{e1a}
\begin{cases}
(\square +1) u = G_1(u,\vv) \\
(\square+1) \vv = G_2(u,\vv).
\end{cases}
\end{align}
In Duhamel formulation, we write
\begin{align*}
\begin{pmatrix}
u(t) \\
\vv(t)
\end{pmatrix}
=\cos(t-T) \langle \nabla \rangle
\begin{pmatrix}
u(T) \\
\vv (T)
\end{pmatrix}
+ \frac{ \sin(t-T)\langle \nabla \rangle}{\langle \nabla \rangle}
\begin{pmatrix}
u_t(T) \\
\vv_t(T)
\end{pmatrix} \\
+ \int_T^t
\frac{\sin (t-s) \langle \nabla \rangle } {\langle \nabla \rangle}
\begin{pmatrix}
G_1(u,\vv) (s) \\
G_2(u, \vv)(s)
\end{pmatrix}
ds.
\end{align*}
Note that
\begin{align*}
\cos(t-T) \langle \nabla \rangle & =\cos t \langle \nabla \rangle \cos T \langle \nabla \rangle
+ \sin t \langle \nabla \rangle \sin T \langle \nabla \rangle \\
\sin(t-T) \langle \nabla \rangle & = \sin t \langle \nabla \rangle \cos T \langle \nabla \rangle
- \cos t \langle \nabla \rangle \sin T \langle \nabla \rangle.
\end{align*}
Therefore
\begin{align}
\begin{pmatrix}
u(t) \\
\vv (t)
\end{pmatrix}
& = \cos t \langle \nabla \rangle \left( \cos T \langle \nabla \rangle
\begin{pmatrix}
u(T) \\
\vv (T)
\end{pmatrix}
- \frac {\sin T \langle \nabla \rangle }{\langle \nabla \rangle}
\begin{pmatrix}
u_t(T) \\
\vv_t(T)
\end{pmatrix}
\right) \label{c1} \\
&\; + \frac {\sin t \langle \nabla \rangle}{ \langle \nabla \rangle}
\left( \langle \nabla \rangle \sin T \langle \nabla \rangle
\begin{pmatrix}
u(T) \\
\vv (T)
\end{pmatrix}
+ \cos T \langle \nabla \rangle
\begin{pmatrix}
u_t(T) \\
\vv_t (T)
\end{pmatrix}
\right) \label{c2} \\
& \quad + \int_T^t \frac { \sin (t-s) \langle \nabla \rangle} { \langle \nabla \rangle}
\begin{pmatrix}
G_1( u, \vv ) (s) \\
G_2( u, \vv ) (s)
\end{pmatrix}
ds.
\end{align}
Consider the linear flow generated by the homogeneous equation
\begin{align*}
\begin{cases}
(\square+1) \tilde u = 0 \\
(\square+1) \tilde \vv = 0
\end{cases}
\end{align*}
with initial data
\begin{align*}
\begin{pmatrix}
\tilde u (0) \\
\tilde \vv (0)
\end{pmatrix}
=
\begin{pmatrix}
\tilde u_0 \\
\tilde \vv_0
\end{pmatrix},
\quad
\begin{pmatrix}
\tilde u_t(0) \\
\tilde \vv_t(0)
\end{pmatrix}
=
\begin{pmatrix}
\tilde u_1 \\
\tilde \vv_1
\end{pmatrix}.
\end{align*}
It is not difficult to check that
\begin{align*}
\begin{pmatrix}
\tilde u(t) \\
\tilde u_t(t)
\end{pmatrix}
=
\begin{pmatrix}
\cos t \jnab \quad \frac{\sin t \jnab}{\jnab} \\
-\jnab \sin t \jnab \quad \cos t \jnab
\end{pmatrix}
\begin{pmatrix}
\tilde u_0 \\
\tilde u_1
\end{pmatrix}
\end{align*}
and similarly
\begin{align*}
\begin{pmatrix}
\tilde \vv(t) \\
\tilde \vv_t(t)
\end{pmatrix}
=
\begin{pmatrix}
\cos t \jnab \quad \frac{\sin t \jnab}{\jnab} \\
-\jnab \sin t \jnab \quad \cos t \jnab
\end{pmatrix}
\begin{pmatrix}
\tilde \vv_0 \\
\tilde \vv_1
\end{pmatrix}.
\end{align*}

Write the linear flow propagator
\begin{align*}
S(t) =
\begin{pmatrix}
\cos t \jnab \quad \frac{\sin t \jnab}{\jnab} \\
-\jnab \sin t \jnab \quad \cos t \jnab
\end{pmatrix}.
\end{align*}

For given data $(f_1,g_1)$, $(f_2,g_2)$, the wave operator problem for
\eqref{e1} amounts to finding the nonlinear solution ($u(t)$, $ \vv(t)$) such that
\begin{align}
S(-T)
\begin{pmatrix}
u(T) \\
u_t(T)
\end{pmatrix}
&\rightarrow
\begin{pmatrix}
f_1 \\
g_1
\end{pmatrix},
\label{c3} \\
S(-T)
\begin{pmatrix}
\vv(T) \\
\vv_t (T)
\end{pmatrix}
&\rightarrow
\begin{pmatrix}
f_2 \\
g_2
\end{pmatrix},
\quad \text{as $T\to \infty$.} \label{c4}
\end{align}
The convergence here takes place in some Sobolev space $H^m$ which will become clear later.
Plugging \eqref{c3}, \eqref{c4} into \eqref{c1}, \eqref{c2}, we get
\begin{align}
\begin{pmatrix}
u(t) \\
\vv (t)
\end{pmatrix}
& = \cos t \jnab \begin{pmatrix} f_1 \\ f_2 \end{pmatrix}
+ \frac {\sin t \jnab} { \jnab} \begin{pmatrix} g_1 \\ g_2 \end{pmatrix} \notag \\
& \quad + \int_\infty^t \frac{\sin(t-s) \jnab} {\jnab}
\begin{pmatrix}
G_1(u, \vv) (s) \\
G_2(u, \vv) (s)
\end{pmatrix}
ds. \label{eq_main}
\end{align}
The system \eqref{eq_main} will be our main object of study. Here is our theorem.

\begin{thm}
Let $m\ge 0$ be an integer.
Let ($f_1^0,f_2^0$), ($g_1^0$, $g_2^0$) be given functions such that
\begin{align} \label{106pm}
\sup_{\xi \in \R^2} \langle \xi \rangle^{m+24}
\sum_{\substack{|\alpha|\le 2 \\ i=1,2}} | \partial_{\xi}^\alpha \hat f_i^0(\xi) |
  +  \sum_{\substack{|\alpha|\le 2 \\ i=1,2}} | \partial_{\xi}^\alpha \hat g_i^0(\xi) |<\infty,
\end{align}
where $\hat f_i^0$, $\hat g_i^0$ are their Fourier transforms.
 Let
\begin{align*}
\begin{pmatrix}
f_1 \\
f_2
\end{pmatrix}
= \epsilon
\begin{pmatrix}
f_1^0 \\
f_2^0
\end{pmatrix},
\quad
\begin{pmatrix}
g_1 \\
g_2
\end{pmatrix}
= \epsilon
\begin{pmatrix}
g_1^0 \\
g_2^0
\end{pmatrix}.
\end{align*}
Then there exists $\epsilon_0>0$, such that for any $0<\epsilon <\epsilon_0$, there exists
a unique global solution to \eqref{eq_main} written in the form
\begin{align*}
\begin{pmatrix}
u(t) \\
\vv(t)
\end{pmatrix}
= \cos t \jnab
\begin{pmatrix}
f_1 \\
f_2
\end{pmatrix}
+ \frac {\sin t \jnab}{\jnab}
\begin{pmatrix}
g_1 \\
g_2
\end{pmatrix}
+ \begin{pmatrix}
z_1(t) \\
z_2(t)
\end{pmatrix},
\end{align*}
where ($z_1(t)$, $z_2(t)$)$\in H^m$, and
\begin{align*}
\sup_{t\ge 0} \langle t \rangle (\| z_1(t) \|_{H^m} + \| z_2(t) \|_{H^m} ) \lsm \epsilon.
\end{align*}

\end{thm}

\begin{rem}
Certainly the global solution for \eqref{eq_main} is also a global solution for
\eqref{e1}. Our constructed solutions scatters in the $H^m$ space at the optimal speed $1/t$, i.e.
\begin{align*}
\Bigl\|
\begin{pmatrix}
u(t) \\
\vv(t)
\end{pmatrix}
- \cos t \jnab
\begin{pmatrix}
f_1 \\
f_2
\end{pmatrix}
- \frac {\sin t \jnab}{\jnab}
\begin{pmatrix}
g_1 \\
g_2
\end{pmatrix}
\Bigr\|_{H^m} \lesssim \frac 1 {\langle t \rangle }, \quad \forall\, t\ge 0.
\end{align*}
Our method of proof also works for the 3D case. In that case we can obtain optimal decay $t^{-3/2}$.
\end{rem}
\begin{rem}
To simplify the presentation, we do not make attempt to optimize the decay assumption \eqref{106pm}
on the final data $(f_1,g_1,f_2,g_2)$. More refined estimates can weaken the assumption \eqref{106pm}
significantly. However we shall not pursue this matter here.
\end{rem}

\section{Reduction of the problem and the iteration scheme} \label{sec_red}
Define the linear solutions
\begin{align}
h_1(t) &= \cos t \jnab f_1 + \frac {\sin t \jnab} {\jnab} g_1, \label{eha}\\
h_2(t) &= \cos t \jnab f_2 + \frac {\sin t \jnab} {\jnab} g_2. \label{ehb}
\end{align}
Let
\begin{align*}
u^{(1)} = u-h_1, \\
\vv^{(1)}= \vv - h_2.
\end{align*}
For $(u^{(1)}, \vv^{(1)})$ we have the system
\begin{align}
(\square+1) u^{(1)} & =
\nabla \cdot \Bigl( (\vv \cdot \nabla)(\vv^{(1)} + h_2)
+ u \nabla (u^{(1)} +h_1) \Bigr) \notag\\
&\quad -\partial_t \Bigl(
(\vv\cdot \nabla)(u^{(1)} +h_1) + u \nabla \cdot (\vv^{(1)} +h_2) \Bigr)
\label{eq845a}
\end{align}
\begin{align}
(\square+1) \vv^{(1)}
&= \nabla \Bigl( (\vv \cdot \nabla)(u^{(1)}+h_1)
+ u \nabla \cdot (\vv^{(1)} +h_2) \Bigr) \notag \\
& \quad
-\partial_t \Bigl(
(\vv\cdot \nabla)(\vv^{(1)}+h_2)
+ u \nabla ( u^{(1)} +h_1) \Bigr) \notag \\
& \qquad - \nabla \Delta^{-1} \nabla ( u \vv). \label{eq845b}
\end{align}
The RHS of \eqref{eq845a}--\eqref{eq845b} still contains quadratic source
terms due to the self interaction of the linear solution
$(h_1,h_2)$. To remove resonances created by these sources, we further
define
\begin{align}
\Phi_1 &
= \int_\infty^t \frac{\sin (t-s) \jnab} {\jnab}
G_1(h_1,h_2)(s) ds, \label{ePhia}\\
\Phi_2 &
= \int_\infty^t \frac{ \sin (t-s) \jnab} {\jnab}
G_2(h_1,h_2)(s)ds. \label{ePhib}
\end{align}
It is clear that
\begin{align*}
(\square+1) \Phi_1 & =G_1(h_1,h_2), \\
(\square+1) \Phi_2 & = G_2(h_1,h_2).
\end{align*}
Also we have the estimate
\begin{align*}
\| \Phi_1(t) \|_{H^k}
+\| (\Phi_1(t))^\prime \|_{H^k}
+\| \Phi_2 (t) \|_{H^k}
+\| ( \Phi_2(t) )^\prime \|_{H^k} \lsm \frac {\epsilon^2} {\langle t \rangle}.
\end{align*}
We then introduce one more change of variables and define
\begin{align*}
\bar u &= u^{(1)} - \Phi_1 = u-h_1-\Phi_1,\\
\bar \vv & = \vv^{(1)} -\Phi_2=\vv - h_2-\Phi_2.
\end{align*}
For the new variables  $(\bar u, \bar \vv)$, we have the system

\begin{align}
(\square+1) \bar u & =
\nabla \cdot \Bigl(  (\vv \cdot \nabla) (\bar \vv + \Phi_2)
+ u \nabla (\bar u + \Phi_1) \Bigr) \notag \\
& \quad -\partial_t \Bigl(
(\vv \cdot \nabla) (\bar u + \Phi_1)
+u \nabla \cdot (\bar \vv + \Phi_2) \Bigr) \notag \\
& \quad + \nabla \cdot \Bigl(
\bigl( (\bar \vv + \Phi_2) \cdot \nabla \bigr) h_2
+ (\bar u + \Phi_1) \nabla h_1 \Bigr) \notag \\
& \quad - \partial_t
\Bigl(
\bigl( (\bar \vv + \Phi_2) \cdot \nabla \bigr) h_1
+ (\bar u +\Phi_1) \nabla \cdot h_2 \Bigr), \label{eq9a} \\
(\square+1) \bar \vv &
= \nabla \Bigl(
(\vv \cdot \nabla) ( \bar u + \Phi_1)
+ u \nabla \cdot (\bar \vv + \Phi_2) \Bigr) \notag \\
& \quad
-\partial_t \Bigl(
(\vv \cdot \nabla) (\bar \vv + \Phi_2)
+ u \nabla ( \bar u + \Phi_1) \Bigr) \notag \\
& \quad
+ \nabla \Bigl(
(\bar \vv + \Phi_2) \cdot \nabla h_1
+ (\bar u + \Phi_1) \nabla \cdot h_2 \Bigr) \notag \\
& \quad -\partial_t
\Bigl( \bigl( (\bar v+ \Phi_2) \cdot \nabla \bigr) h_2
+ (\bar u + \Phi_1) \nabla h_1 \Bigr) \notag \\
& \qquad - \nabla \Delta^{-1} \nabla ( u \vv), \label{eq9b}
\end{align}
and
\begin{align}
u & = \bar u + \Phi_1 +h_1, \notag \\
\vv & = \bar \vv + \Phi_2 +h_2. \label{eq9c}
\end{align}
The system \eqref{eq9a}--\eqref{eq9c} in the variables
$(\bar u, \bar \vv)$ no longer contains resonant terms.
We will construct a smooth global solution to this system.

Let $\chi \in C_c^\infty(\R)$ be such that
$\chi(t) =1$ for $0\le t \le 1$ and $\chi(t)=0$ for
$t\ge 2$. For integers $n\ge 0$, define
\begin{align*}
\chi_n(t) = \chi(2^{-n} t).
\end{align*}
It is clear that $\chi_n$ is localized to the regime $t \lsm 2^n$.

Let $\bar u^{(0)}=0$, $\bar \vv^{(0)}=0$. For $n\ge 1$ we inductively define

\begin{align}
(\square +1) \bar u^{(n)} & = \chi_n(t) \nabla \cdot
\Bigl( (\vv^{(n-1)} \cdot \nabla)(\bar \vv^{(n)}) + u^{(n-1)} \nabla \bar u^{(n)} \Bigr) \notag \\
& \quad - \chi_n(t) \partial_t \Bigl( (\vv^{(n-1)} \cdot \nabla) (\bar u^{(n)} ) + u^{(n-1)} \nabla \cdot ( \bar \vv^{(n)}) \Bigr) \notag \\
& \quad + \nabla \cdot \Bigl( (\vv^{(n-1)} \cdot \nabla) \Phi_2 + u^{(n-1)} \nabla \Phi_1 \Bigr) \notag \\
& \quad - \partial_t \Bigl( (\vv^{(n-1)} \cdot \nabla) \Phi_1 + u^{(n-1)} \nabla \cdot \Phi_2 \Bigr) \notag \\
& \quad + \nabla \cdot \Bigl(
\bigl( (\bar \vv^{(n-1)} + \Phi_2) \cdot \nabla \bigr) h_2 + (\bar u^{(n-1)} + \Phi_1) \nabla h_1 \Bigr) \notag \\
& \quad - \partial_t
\Bigl( \bigl( (\bar v^{(n-1)} + \Phi_2 ) \cdot \nabla \bigr) h_1 + ( \bar u^{(n-1)} + \Phi_1 ) \nabla \cdot h_2 \Bigr), \label{eq10a}
\end{align}
and
\begin{align}
(\square +1) \bar \vv^{(n)} & = \chi_n(t) \nabla \Bigl( (\vv^{(n-1)} \cdot \nabla)(\bar u^{(n)}) + u^{(n-1)} \nabla \cdot \bar \vv^{(n)} \Bigr) \notag \\
& \quad - \chi_n(t) \partial_t
\Bigl( (\vv^{(n-1)} \cdot \nabla) ( \bar \vv^{(n)}) + u^{(n-1)} \nabla ( \bar u^{(n)} ) \Bigr) \notag \\
& \quad + \nabla \Bigl( (\vv^{(n-1)} \cdot \nabla) \Phi_1 + u^{(n-1)} \nabla \cdot \Phi_2 \Bigr) \notag \\
& \quad -\partial_t \Bigl( (\vv^{(n-1)} \cdot \nabla ) \Phi_2 + u^{(n-1)} \nabla \Phi_1 \Bigr) \notag \\
& \quad + \nabla \Bigl( (\bar \vv^{(n-1)} + \Phi_2) \cdot \nabla h_1 + ( \bar u^{(n-1)} + \Phi_1) \nabla \cdot h_2 \Bigr) \notag \\
& \quad - \partial_t \Bigl(
\bigl( (\bar \vv^{(n-1)} + \Phi_2) \cdot \nabla \bigr) h_2 + (\bar u^{(n-1)} + \Phi_1) \nabla h_1 \Bigr) \notag \\
& \qquad - \nabla \Delta^{-1} \nabla \cdot( u^{(n-1)} \vv^{(n-1)} ). \label{eq10b}
\end{align}

Here
\begin{align*}
u^{(n-1)} & = \bar u^{(n-1)} + h_1+\Phi_1, \\
\vv^{(n-1)} & = \bar \vv^{(n-1)} +h_2+\Phi_2.
\end{align*}

In Section \ref{sec_solv} we will study the system \eqref{eq10a}--\eqref{eq10b} and its solvability.

\section{sharp phase estimates for $\Phi_1$, $\Phi_2$} \label{sec_phi}
In this section we estimate $(\Phi_1,\Phi_2)$. By \eqref{ePhia}--\eqref{ePhib},
\eqref{eha}--\eqref{ehb} and \eqref{e1a}, we have
\begin{align*}
\Phi_1 & =\int_\infty^t \frac {\sin (t-s) \jnab} {\jnab} \Bigl( \nabla \cdot \left( (h_2(s) \cdot \nabla ) h_2 (s)
+ h_1(s) \nabla h_1(s) \right) \\
& \qquad -\partial_s \left( h_2(s) \cdot \nabla h_1(s) + h_1(s) \nabla \cdot h_2(s) \right) \Bigr) ds,
\end{align*}
where we recall
\begin{align*}
h_1(s) & = \cos s \jnab f_1 + \frac {\sin s \jnab} {\jnab} g_1,\\
h_2(s) & = \cos s \jnab f_2 + \frac { \sin s \jnab } {\jnab} g_2.
\end{align*}
By Fourier transform, one can write down the explicit form of $\hat \Phi_1$ in terms of
the sum of several bilinear integrals involving the functions $(f_1,g_1,f_2,g_2)$. There
are altogether sixteen such terms. To simplify the notations, we shall only write
\begin{align*}
\hat \Phi_1(t, \xi) \sim e^{\pm it \langle \xi \rangle} \cdot \langle \xi \rangle^{-1} \int_\infty^t \int_{\R^2}
e^{-is ( \langle \xi \rangle \pm \langle \xi -\eta \rangle \pm \langle \eta \rangle ) }
\hat f(\xi -\eta) \hat g(\eta) d\eta ds,
\end{align*}
where
\begin{align*}
\hat f (\xi) &= \left( C_1 \xi^{\alpha_1} + C_2 \langle \xi \rangle^{r_1} \right) \tilde f (\xi),\\
\hat g (\xi) & = \left( C_3 \xi^{\alpha_2} + C_4 \langle \xi \rangle^{r_2} \right) \tilde g (\xi),
\end{align*}
and $|\alpha_1|\le 2$, $|\alpha_2|\le 2$, $\alpha_1$, $\alpha_2$ are multi-indices, also
$r_1=\pm 1$, $r_2 = \pm 1$, $C_1,\cdots, C_4$ are constants, $\tilde f$, $\tilde g$ can be any
one choice from $(\hat f_1, \hat g_1, \hat f_2,\hat g_2)$.

Let $m\ge 0$ be any integer. We compute
\begin{align*}
 & \| \langle t \rangle \Phi_1(t) \|_{H^{m+2}} + \| \langle t \rangle \partial_t \Phi_1 (t) \|_{H^{m+2}} \\
\lesssim  & \| \langle \xi \rangle^{m+2} \jt \int_\infty^t \int_{\R^2}
e^{is(\langle \xi \rangle \pm \langle \eta \rangle \pm \langle \xi-\eta \rangle )}
\hat f(\eta) \hat g(\xi-\eta) d\eta ds \|_{L_\xi^2} \\
\lesssim & \| \langle \xi \rangle^{m+2} \jt \int_\infty^t \int_{\R^2}
e^{is(\langle \xi \rangle \pm \langle \frac {\xi} 2+ \eta \rangle \pm \langle \frac{\xi}2-\eta \rangle )}
\hat f(\frac{\xi}2+\eta) \hat g(\frac{\xi}2-\eta) d\eta ds \|_{L_\xi^2},
\end{align*}
where in the last step we made a change of variable $\eta \to \frac {\xi} 2 + \eta$.

The main result of this section is the following

\begin{thm} \label{thm_51}
For any integer $m\ge 0$, we have the estimate
\begin{align*}
  & \| \langle t \rangle \Phi_1(t) \|_{H^{m+2}} + \| \langle t \rangle \partial_t \Phi_1(t) \|_{H^{m+2}} \\
  \le & C \sup_{\xi \in \mathbb R^2} \langle \xi \rangle^{m+24}
  \Bigl( \sum_{\substack{|\alpha|\le 2 \\ i=1,2}} | \partial_{\xi}^\alpha \hat f_i(\xi) |
  +  \sum_{\substack{|\alpha|\le 2 \\ i=1,2}} | \partial_{\xi}^\alpha \hat g_i(\xi) | \Bigr),
\end{align*}
where $C$ is an absolute constant. $(\hat f_1, \hat g_1, \hat f_2, \hat g_2)$ are Fourier transforms of
the final data $(f_1,g_1, f_2,g_2)$ respectively. Same estimates also hold for $\Phi_2$.
\end{thm}

Theorem \ref{thm_51} will follow from the following

\begin{prop} \label{prop_51}
For any integer $m\ge 0$, we have
\begin{align}
&\| \langle \xi \rangle^{m+2} \jt \int_\infty^t \int_{\R^2}
e^{is(\langle \xi \rangle \pm \langle \frac {\xi} 2+ \eta \rangle \pm \langle \frac{\xi}2-\eta \rangle )}
h(\xi, \eta) d\eta ds \|_{L_\xi^2}  \notag \\
\lsm  &
 \sup_{\tilde \xi, \tilde \eta \in \R^2}
\langle |\tilde \xi| + |\tilde \eta|\rangle^{m+21}
( |h(\tilde \xi,\tilde \eta)| + | (\partial_{\tilde \eta} h)(\tilde \xi, \tilde \eta) |
+ |(\partial^2_{\tilde \eta} h)(\tilde \xi,\tilde \eta) |) \label{e_tm_128}
\end{align}

\end{prop}
Note that it is enough to prove Proposition \ref{prop_51} for the case $m=0$.
We defer the proof of Proposition \ref{prop_51} to the end of this section.
For now we assume Proposition \ref{prop_51} holds and complete the
\begin{proof}[Proof of Theorem \ref{thm_51}]
It suffices to show the case $m=0$.
By Proposition \ref{prop_51} and the brief computation preceding Theorem \ref{thm_51},
we only need to bound the quantity
\begin{align*}
\sup_{\tilde \xi, \tilde \eta \in \R^2}
\langle |\tilde \xi| + |\tilde \eta|\rangle^{21}
( |h(\tilde \xi,\tilde \eta)| + | (\partial_{\tilde \eta} h)(\tilde \xi, \tilde \eta) |
+ |(\partial^2_{\tilde \eta} h)(\tilde \xi,\tilde \eta) |),
\end{align*}
where $h(\xi,\eta) \sim \hat f(\frac \xi 2 +\eta) \hat g(\frac \xi 2 -\eta)$, and
$\hat f (\xi)$, $\hat g(\xi)$ depends only on the final data $(f_1,g_1,f_2,g_2)$. An
elementary case by case analysis then yields the result.
\end{proof}

The rest of this section is devoted to the proof of Proposition \ref{prop_51}.
We begin by establishing some elementary lemmas.
\begin{lem} \label{lem_1}
 For any $a$, $b\in \R^2$, we have
\begin{align}
 \left| \frac{a+b}{\langle a+b \rangle} + \frac {a-b}{\langle a-b \rangle} \right|
\gtrsim \frac {|a|}{\langle |a| + |b| \rangle^3}. \label{e_lem1_1}
\end{align}
\end{lem}
\begin{proof}
 If $a=0$ or $b=0$ then \eqref{e_lem1_1} holds trivially. Now assume both $a$ and $b$ are nonzero
vectors. Denote the vector on the LHS of \eqref{e_lem1_1} as $X$. We consider several cases.
WOLOG we may assume $a\cdot b\ge 0$.

\texttt{Case 1}: $|a|\ge |b|$. Then $(a+b)\cdot (a-b)\ge 0$. Therefore we have
\begin{align*}
 |X| \gtrsim \frac{|a+b|}{\langle a+b \rangle} \gtrsim \frac{|a|}{\langle |a|+|b| \rangle}.
\end{align*}

\texttt{Case 2}: $|a| \le |b|$. We decompose the vector $a$ as
\begin{align*}
 a=a_1+a_2, \quad \text{where $a_1\parallel b$, $a_2 \perp b$.}
\end{align*}
Then we have
\begin{align}
 |X| = \sqrt{ |\frac{a_1+b}{\langle a+b \rangle } + \frac{a_1-b} {\langle a -b \rangle } |^2
+ | a_2 ( \frac 1{\langle a+b \rangle } + \frac 1{\langle a-b \rangle }) |^2}.
\label{e_lem1_2}
\end{align}
Consider two subcases.

\texttt{Subcase 2a}: $|a_2| \ge \frac 13 |a|$. Clearly by \eqref{e_lem1_2}, we get
\begin{align*}
 |X| \gtrsim |a_2| \cdot \frac 1 {\langle |a|+|b| \rangle} \gtrsim \frac {|a|}{\langle |a|+|b| \rangle}.
\end{align*}

\texttt{Subcase 2b}: $|a_1|\ge \frac 13 |a|$. Since $a \cdot b \ge 0$, we can write
$a_1=\lambda b$, with $\frac 1 3 \cdot \frac {|a|}{|b|} \le \lambda \le \frac {|a|}{|b|}$. By
\eqref{e_lem1_2} and the fact $a_1 \parallel b$, we obtain
\begin{align*}
 |X| &\gtrsim \left| \frac {(\lambda+1) b}{\langle a_2+(\lambda+1)b \rangle } +
\frac {(\lambda-1) b}{\langle a_2+(\lambda-1)b \rangle } \right| \\
& \gtrsim \left| \sqrt{ \frac{(\lambda+1)^2 |b|^2} {1+|a_2|^2+(\lambda+1)^2 |b|^2 }}
-\sqrt{ \frac{(\lambda-1)^2 |b|^2} {1+|a_2|^2+(\lambda-1)^2 |b|^2 }} \right| \\
& \gtrsim \frac{ \frac{(1+|a_2|^2) \cdot 4 \lambda |b|^2 } { (1+|a_2|^2+(\lambda+1)^2 |b|^2)
\cdot (1+|a_2|^2+(1-\lambda)^2|b|^2)} }
{ \sqrt{ \frac{(\lambda+1)^2|b|^2}{1+|a_2|^2+(\lambda+1)|b|^2}}
+\sqrt{ \frac{(\lambda-1)^2|b|^2}{1+|a_2|^2+(\lambda-1)|b|^2}}} \\
&\gtrsim \frac{|a|}{\langle a+b \rangle^2 \cdot \langle a-b\rangle^2} \cdot \langle a-b \rangle \\
& \gtrsim \frac{|a|}{\langle |a|+|b| \rangle^3}.
\end{align*}
The lemma is proved.

\end{proof}

\begin{lem} \label{lem_2}
 For any $A$, $B \in \R$, we have
\begin{align}
 \sqrt{1+A^2} + \sqrt{1+B^2} - \sqrt{1+(A+B)^2} \gtrsim \frac 1 {\sqrt{1+A^2}+\sqrt{1+B^2}}.
\label{e_lem_2_0}
\end{align}

\end{lem}
\begin{proof}
 This is almost trivial. We compute
\begin{align*}
 \text{LHS of \eqref{e_lem_2_0}} & = \frac{2+A^2+B^2+2\sqrt{1+A^2}\cdot \sqrt{1+B^2} - (1+A^2+B^2+2AB)}
{\sqrt{1+A^2}+\sqrt{1+B^2}+\sqrt{1+(A+B)^2}} \\
&\gtrsim \frac {1}{\sqrt{1+A^2}+\sqrt{1+B^2}}.
\end{align*}

\end{proof}

Now define
\begin{align*}
 a_1(\xi,\eta)&= \langle \xi \rangle - \langle \frac {\xi} 2 + \eta \rangle - \langle \frac {\xi} 2 -\eta \rangle, \\
a_2(\xi,\eta)&= \langle \xi \rangle + \langle \frac {\xi} 2 + \eta \rangle + \langle \frac {\xi} 2 -\eta \rangle, \\
a_3(\xi,\eta)&= \langle \xi \rangle - \langle \frac {\xi} 2 + \eta \rangle + \langle \frac {\xi} 2 -\eta \rangle, \\
a_4(\xi,\eta)&= \langle \xi \rangle + \langle \frac {\xi} 2 + \eta \rangle - \langle \frac {\xi} 2 -\eta \rangle. \\
\end{align*}

\begin{lem} \label{lem_3}
For any $\xi$, $\eta \in \R^2$, we have
 \item \begin{align}
       \frac 1 {\langle |\xi| + |\eta| \rangle} \lsm |a_i(\xi,\eta) | &
\lsm  {\langle |\xi| + |\eta| \rangle}, \quad\,\forall\, 1\le i \le 4.
\label{e_lem_3_1} \\
|\nabla_\eta a_i (\xi, \eta) | & \gtrsim
\begin{cases}
\frac {|\eta|}{ \langle |\xi|+|\eta| \rangle^3}, \quad i=1,2  \\
\frac {|\xi|}{\langle |\xi|+|\eta| \rangle^3}, \quad i=3,4
\end{cases}
\label{e_lem_3_3} \\
|\partial_\eta^k a_i(\xi,\eta) | & \lsm 1, \quad \forall\, 1\le i\le 4, \, k\ge 1.
\label{e_lem_3_5}
       \end{align}
If $i=3,4$, then
\begin{align}
 |\partial_\eta^k a_i (\xi, \eta) | \lsm |\xi|, \quad\forall\,1\le k\le 4, \label{e_lem_3_7} \\
\left|\partial_{\eta}\left( \frac {\nabla_\eta a_i(\xi,\eta)}{|\nabla_\eta a_i(\xi,\eta)|^2} \right)
\right| \lsm \frac{\langle |\xi|+|\eta| \rangle^6} {|\xi|}, \quad \forall\, \xi \ne 0, \label{e_lem_3_9} \\
\left|\partial_{\eta}^2 \left( \frac {\nabla_\eta a_i(\xi,\eta)}{|\nabla_\eta a_i(\xi,\eta)|^2} \right)
\right| \lsm \frac{\langle |\xi|+|\eta| \rangle^9} {|\xi|}, \quad \forall\, \xi\ne 0. \label{e_lem_3_11}
\end{align}
\end{lem}
\begin{proof}
The lower bound in \eqref{e_lem_3_1} follows from Lemma \ref{lem_2}. The upper bound in
\eqref{e_lem_3_1} is trivial. If $i=1,2$, then \eqref{e_lem_3_3} follows from Lemma \ref{lem_1}
with $a=\eta$, $b=\frac{\xi}2$. Similarly the case $i=3,4$ follows from Lemma \ref{lem_1}
with $a=\frac {\xi}2$, $b=\eta$. Next observe that $|\partial_x^k ( \langle x \rangle ) | \lsm
\langle x \rangle^{-(k-1)}$ for any $k\ge 1$. Therefore \eqref{e_lem_3_5} holds.
Next we show \eqref{e_lem_3_7}. This is a matter of direct calculation.
Define $a(x) = \langle x \rangle $. Then for $i=3,4$, $1\le k \le 4$,
\begin{align*}
 |\partial_\eta^k a_i(\xi, \eta) | & = |(\partial_x^k a)(\eta+\frac {\xi} 2)
-(\partial_x^k a)(\eta -\frac{\xi} 2) | \\
& \lsm \| \partial_x^{k+1} a \|_{\infty} \cdot |\xi| \\
& \lsm |\xi|,
\end{align*}
where in the last inequality we have used again the fact that $|\partial_x^k ( \langle x \rangle ) | \lsm
\langle x \rangle^{-(k-1)}$. This finishes the proof of \eqref{e_lem_3_7}. Finally we show
\eqref{e_lem_3_9} and \eqref{e_lem_3_11}. By direct differentiation, disgarding inessential coefficients and
keeping only the main terms, we have
\begin{align*}
 \partial_\eta \left( \frac{\nabla_\eta a_i(\xi,\eta)}{ |\nabla a_i (\xi,\eta) |^2} \right)
& \sim \frac{(\partial_\eta a_i(\xi,\eta))^2 \partial_\eta^2 a_i (\xi,\eta)} { |\nabla_\eta a_i(\xi,\eta)|^4},\\
\partial_\eta^2 \left( \frac{\nabla_\eta a_i(\xi,\eta)}{ |\nabla a_i (\xi,\eta) |^2} \right)
& \sim \frac{ \partial_\eta^3 a_i(\xi,\eta) \cdot (\partial_\eta a_i(\xi,\eta))^2
+ (\partial_\eta^2 a_i(\xi,\eta))^2 \cdot \partial_\eta a_i(\xi,\eta)} {|\nabla_\eta a_i(\xi,\eta)|^4} \\
& \quad + \frac{(\partial_\eta a_i(\xi,\eta))^3 \cdot (\partial_\eta^2 a_i(\xi,\eta))^2 }{ |\nabla_\eta a_i(\xi,\eta)|^6}.
\end{align*}
Therefore by \eqref{e_lem_3_3} and \eqref{e_lem_3_7}, we obtain
\begin{align*}
 \left| \partial_\eta \left( \frac{\nabla_\eta a_i(\xi,\eta)}{|\nabla_\eta a_i(\xi,\eta)|^2 } \right)
\right| & \lsm \frac{|\partial_\eta^2 a_i(\xi,\eta)|}{|\nabla_\eta a_i(\xi,\eta)|^2} \lsm
|\xi| \cdot \frac {\langle |\xi|+|\eta| \rangle^6} {|\xi|^2} \\
&\lsm \frac{\langle |\xi|+|\eta| \rangle^6}{|\xi|},
\end{align*}
and
\begin{align*}
\left| \partial_\eta^2 \left( \frac{\nabla_\eta a_i(\xi,\eta)}{|\nabla_\eta a_i(\xi,\eta)|^2 } \right)
\right| & \lsm  \frac{\langle |\xi|+|\eta \rangle^6}{|\xi|} + |\xi|^2 \cdot
\frac{\langle |\xi|+|\eta \rangle^9}{|\xi|^3} \\
& \lsm \frac{\langle |\xi|+|\eta \rangle^9}{|\xi|}.
\end{align*}

\end{proof}

\begin{lem}\label{lem_101}
Let $a_j(\xi,\eta)$ be as defined as above. Define
\begin{align*}
F_i(s,\xi)=\int_{\R^2} e^{is(a_i(\xi,\eta)-a_i(\xi,0))}h(\eta) d\eta.
\end{align*}
Then we have
\begin{align*}
|F_i(s,\xi)|&\lsm \frac 1{\langle s\rangle^{1-}}\lrx ^{12}(1+|\xi|^{-2})\\
&(\sum_{|\alpha|\le 2}\int \lre^{12}|\partial^{\alpha} h(\eta)|d\eta+\sum_{|\alpha|\le 2}
\sup_{|\eta|\le 2}|\partial^{\alpha}h(\eta)|)\\
|\partial_s F_i(s,\xi)|&\lsm \frac 1{\langle s\rangle^{2-}}\lrx ^{12}(1+|\xi|^{-2})\\
&(\sum_{|\alpha|\le 2}\int \lre^{12}|\partial^{\alpha} h(\eta)|d\eta+\sum_{|\alpha|\le 2}
\sup_{|\eta|\le 2 }|\partial^{\alpha}h(\eta)|)
\end{align*}
\end{lem}
\begin{proof}
We first estimate $\partial_s F_i$. 
Without loss of generality, we also assume $s>1$.

We first consider $\partial_s F_1, \partial_s F_2$, where we shall use the following bound for $a_i$, $i=1,2$,
\begin{align}\label{z1}
\begin{cases}
|\nabla a_i|\gtrsim \frac{|\eta|}{\langle |\eta|+|\xi|\rangle^3}\\
|\partial^ k a_i|\lsm 1, \ \mbox{ for } k\ge 1,\\
|a_i(\xi,\eta)-a_i(\xi,0)|\lsm |\eta|^2.
\end{cases}
\end{align}
Since for the phase function $a_i(\xi,\eta)-a_i(\xi,0)$, the stationary phase point occurs
at $\eta=0$, we introduce the cutoff function separate the near-origin regime and others.

Let $\psi(\eta)$ be a smooth cutoff function such that
\begin{align}
\psi(\eta)=\begin{cases}
1, & |\eta|\le 1,\\
0, & |\eta|>2
\end{cases}
\end{align}
We split $\partial_s F_i$, $i=1,2$ as follows
\begin{align}
\partial_s F_i(s,\xi)&=\int_{\R^2}(a_i(\xi,\eta)-a_i(\xi,0))e^{is(a_i(\xi,\eta)-a_i(\xi,0))} h(\eta) d\eta\notag\\
&=\int_{\R^2}(a_i(\xi,\eta)-a_i(\xi,0))e^{is(a_i(\xi,\eta)-a_i(\xi,0))} h(\eta) \psi(\eta) d\eta\label{lem_101_a}\\
&+\int_{\R^2}(a_i(\xi,\eta)-a_i(\xi,0))e^{is(a_i(\xi,\eta)-a_i(\xi,0))} h(\eta)(1-\psi(\eta))d\eta\label{lem_101_b}
\end{align}

For \eqref{lem_101_b}, we do integration by parts twice. This gives
\begin{align*}
\eqref{lem_101_b}=&\frac 1{s^2}\int_{\R^2}\nabla\cdot \biggl(\frac{\nabla a_i}{|\nabla a_i|^2}
\nabla\cdot\biggl(\frac{\nabla a_i}{|\nabla a_i|^2}(a_i(\xi,\eta)-a_i(\xi,0))h(\eta)
(1-\psi(\eta))\biggl)\biggl)\\
 &\qquad \qquad e^{is(a_i(\xi,\eta)-a_i(\xi,0))} d\eta.
\end{align*}
Using the product rule and the estimates \eqref{z1}, we obtain
\begin{align*}
|\eqref{lem_101_b}|&\lsm \frac 1{s^2}\int_{|\eta|\ge 1}
(1+|\eta|^{-4})(|\xi|+|\eta|)^{12}(\sum_{|\alpha|\le 2}|\partial ^{\alpha} h(\eta)|)d\eta\\
&\lsm \frac 1{s^2}\int_{|\eta|\ge 1}\langle|\xi|+|\eta|\rangle^{12}(\sum_{|\alpha|\le 2}|
\partial^{\alpha} h(\eta)|) d\eta\\
&\lsm \frac 1{s^2}\lrx^{12}\sum_{|\alpha|\le 2}
\int_{\R^2} \lre ^{12}
|\partial^{\alpha} h(\eta)| d\eta.
\end{align*}
We now turn to estimating \eqref{lem_101_a}. To this end, we further write it into the following
\begin{align}
\eqref{lem_101_a}&=\int (a_i(\xi,\eta)-a_i(\xi,0))e^{is(a_i(\xi,\eta)-a_i(\xi,0))}
h(\eta)\psi(\eta)\psi(\eta s^{10}) d\eta\label{a1}\\
&+\int (a_i(\xi,\eta)-a_i(\xi,0))e^{is(a_i(\xi,\eta)-a_i(\xi,0))} h(\eta)\psi(\eta)(1-\psi(\eta s^{10})) d\eta.\label{a2}
\end{align}
For \eqref{a1}, we simply use \eqref{z1} to get
\begin{align}
\eqref{a1}&\lsm \int_{|\eta|\le 2 s^{-10}}|\eta||h(\eta)|\psi(\eta)|d\eta\notag\\
&\lsm s^{-10}\int |h(\eta)|d\eta.\label{finish_2}
\end{align}
For \eqref{a2}, we use integration by parts twice and \eqref{z1} to get
\begin{align*}
|\eqref{a2}|&\le \frac 1{s^2}\int \biggl|
\nabla\cdot\biggl(\frac{\nabla a_i}{|\nabla a_i|^2} \nabla \cdot\biggl(\frac{\nabla a_i}
{|\nabla a_i|^2}(a_i(\xi,\eta)-a_i(\xi,0))h(\eta)\psi(\eta)(1-\psi(\eta s^{10}))\biggl)\biggl)\biggl| d\eta\\
&\lsm \frac 1{s^2}\biggl\{
\int |a_i(\xi,\eta)-a_i(\xi,0)||\nabla a_i|^{-4}|h(\eta)||\psi(\eta)||1-\psi(\eta s^{10})|d\eta\\
&+\int_{|\eta|\le 2} |a_i(\xi,\eta)-a_i(\xi,0)||\nabla a_i|^{-3}(\sum_{|\alpha|\le 1}|\partial^{\alpha} h(\eta)|)(|1-\psi(\eta s^{10})|+s^{10}|\nabla\psi(\eta s^{10})|) d\eta\\
&+\int_{|\eta|\le 2} |a_i(\xi,\eta)-a_i(\xi,0)||\nabla a_i|^{-2}(\sum_{|\alpha|\le 2}|\partial^{\alpha} h(\eta)|)\\
&\qquad\qquad\times(|1-\psi(\eta s^{10})|+s^{10}|\nabla \psi(\eta s^{10})|+s^{20}|\Delta \psi(\eta s^{10})|) d\eta\\
&+\int_{|\eta|\le 2} |\nabla a_i||\nabla a_i|^{-3}|h(\eta)| d\eta\\
&+\int_{|\eta|\le 2} |\nabla a_i||\nabla a_i|^{-2}(\sum_{|\alpha|\le 1}|\partial^{\alpha}h(\eta)|)
(|1-\psi(\eta s^{10})|+s^{10}|\nabla \psi(\eta s^{10})|) d\eta.\biggr\}
\end{align*}
Using estimates \eqref{z1}, we further bound \eqref{a2} as
\begin{align*}
|\eqref{a2}|&\lsm \frac 1{s^2}\biggl\{ \int _{s^{-10}\le |\eta|\le 2}
|\eta|^{-2}\langle |\eta|+|\xi|\rangle ^{12}|h(\eta)|d\eta\\
&\quad+\int_{s^{-10}\le |\eta|\le 2}|\eta|^{-1}
\langle|\eta|+|\xi|\rangle^9(\sum_{|\alpha|\le 1}|\partial^{\alpha} h(\eta)|)d\eta\\
&\quad+\int_{|\eta|\sim s^{-10}}\langle|\eta|+|\xi|\rangle^9 \eta^{-1} s^{10}(\sum_{|\alpha|\le 1}
|\partial^{\alpha} h(\eta)|) d\eta\\
&\quad+\int_{s^{-10}\le |\eta|\le 2}\langle|\eta|+|\xi|\rangle^6(\sum_{|\alpha|\le 2}|\partial ^{\alpha}h(\eta)|)d\eta\\
&\quad+\int_{|\eta|\sim s^{-10}}\langle|\eta|+|\xi|\rangle^6(\sum_{|\alpha|\le 2}|\partial^{\alpha } h(\eta)|)s^{20} d\eta\biggl\}\\
&\lsm \frac 1{s^2}\log s\lrx^{12}\sum_{|\alpha|\le 2}\sup_{|\eta|\le 2}|\partial^{\alpha} h(\eta)|\\
&\lsm \frac {\lrx^{12}}{s^{2-}}\sum_{|\alpha|\le 2}\sup_{|\eta|\le 2}|\partial^{\alpha} h(\eta)|.
\end{align*}
Collecting the estimates above we obtain
\begin{align}\label{est_12}
|\partial_s F_i(s,\xi)|\lsm \frac 1{s^{2-}}\lrx^{12}
\sum_{|\alpha|\le 2}(\int \lre^{12}|\partial^{\alpha} h(\eta)| d\eta+
\sup_{|\eta|\le 2}|\partial^{\alpha} h(\eta)|).
\end{align}
for $i=1,2$.

Now it remains to obtain the estimates for $F_i$, $i=3,4$. We will use the following estimates for $a_i$,
\begin{align}\label{z2}
\begin{cases}
|\nabla a_i|\gtrsim \frac{|\xi|}{\langle |\eta|+|\xi|\rangle^3}\\
|\partial^ k a_i|\lsm 1,\ k\ge 1,\\
\ |a_i(\xi,\eta)-a_i(\xi,0)|\lsm |\eta|.
\end{cases}
\end{align}
We will estimate $\partial_s F_i$ simply by integration by parts twice. We have
\begin{align*}
|\partial_s F_i(s,\xi)|&=\biggl|\int (a_i(\xi,\eta)-a_i(\xi,0)) e^{is(a_i(\xi,\eta)-a_i(\xi,0))} h(\eta) d\eta\biggl|\\
&=\frac 1{s^2}\biggl|\int\nabla\cdot\biggl(\frac{\nabla a_i}{|\nabla a_i|^2}\nabla\cdot\biggl(
\frac{\nabla a_i}{|\nabla a_i|^2}(h(\eta)(a_i(\xi,\eta)-a_i(\xi,0)))\biggl)\biggl) e^{is(a_i(\xi,\eta)-a_i(\xi,0))} d\eta\biggl|\\
&\lsm \frac 1{s^2}\biggl\{\int |\nabla a_i|^{-4} h(\eta)|\eta| d\eta\\
&\qquad+\int |\nabla a_i|^{-3}\sum_{|\alpha|\le 1}|\partial^{\alpha} h(\eta)|(|\eta|+|\nabla a_i|) d\eta\\
&\qquad+\int |\nabla a_i|^{-2}\sum_{|\alpha|\le 2}|\partial^{\alpha} h(\eta)|\lre d\eta\biggl\}\\
&\lsm \frac 1{s^2}\int |\xi|^{-2}\langle |\eta|+|\xi|\rangle^{12}\sum_{|\alpha|\le 2}
|\partial^{\alpha}h(\eta)|d\eta\\
&\lsm \frac 1{s^2}|\xi|^{-2}\lrx^{12}\sum_{|\alpha|\le 2}\int \partial^{\alpha}h(\eta)\lre^{12} d\eta.
\end{align*}
This finishes the estimates for $\partial_s F_i$, $i=3,4$. Finally we make a remark about the decay of $F_i$. Indeed, the decay of $F_i$ can be obtained by slightly modifying the above computation, namely, replacing the integration by parts twice with IBP once. We omit the details. The proof of this lemma is completed.

\end{proof}

As a consequence of this lemma, we have
\begin{lem}\label{lem_102}
Assume the function $h=h(\xi,\eta)$ is such that
\begin{align*}
 \sum_{|\alpha|\le 2} \int_{\R^2} \langle \eta \rangle^{12} |\partial^\alpha h(\xi,\eta)| d\eta
+\|\partial^\alpha h(\xi,\eta)\|_\infty <\infty.
\end{align*}
Then for $i=1,2,3,4$ and any $\xi\neq 0$, we have
\begin{align}
\int_ t^{\infty}\int_{\R^2} e^{isa_i(\xi,\eta)} h(\xi,\eta) d\eta ds=-
\int_{\R^2} \frac{e^{it a_i(\xi,\eta)}}{ia_i(\xi,\eta)} h(\xi,\eta) d\eta .
\end{align}
\end{lem}
\begin{proof}
Define
\begin{align*}
F(s,\xi)=\int e^{is(a_i(\xi,\eta)-a_i(\xi,0))} h(\xi,\eta) d\eta.
\end{align*}
We write
\begin{align*}
&\quad\int_t^{\infty}\int e^{is a_i(\xi,\eta)} h(\xi,\eta) d\eta ds
=\int_t^{\infty} e^{isa_i(\xi,0)} F_i(s,\xi) ds\\
&=\int_t^{\infty}\frac 1{ia(\xi,0)}\int_t^{\infty} \partial_s(e^{isa_i(\xi,0)}F_i(s,\xi))ds
-\frac 1{ia_i(\xi,0)}\int_t^{\infty} e^{isa_i(\xi,0)}\partial_s F_i(s,\xi) ds\\
&=-\frac 1{ia_i(\xi,0)}e^{ita_i(\xi,0)}F_i(t,\xi)-\frac 1{ia_i(\xi,0)}\int_t^{\infty}
e^{isa_i(\xi,0)}\partial_sF_i(s,\xi) ds.
\end{align*}
Now due to the decay of $\partial_s F_i$, the integral in the above is absolutely convergent. Therefore we have
\begin{align*}
&\quad\int_t^{\infty}\int e^{isa_i(\xi,\eta)} h(\xi,\eta) d\eta ds\\
&=\lim_{\eps\to 0^+}\biggl[-\frac 1{ia_i(\xi,0)-\eps} e^{ita_i(\xi,0)-\eps t} F_i(t,\xi)
-\frac 1{ia(\xi,0)-\eps}\int_t^{\infty} e^{isa_i(\xi,0)-\eps s}\partial_s F_i(s,\xi) ds\biggl]\\
&=\lim_{\eps\to 0^+}\int_t^{\infty} e^{isa_i(\xi,0)-\eps s}F_i(s,\xi) ds\\
&=\lim_{\eps\to 0^+}\int_t^{\infty}\int e^{isa_i(\xi,\eta)-\eps s}h(\xi,\eta) d\eta ds\\
&=\lim_{\eps\to 0^+}-\int \frac{e^{ita_i(\xi,\eta)-\eps t}}{ia_i(\xi,\eta)-\eps} h(\xi,\eta) d\eta\\
&=-\int \frac{e^{it a_i(\xi,\eta)}}{ia_i(\xi,\eta)} h(\xi,\eta) d\eta.
\end{align*}

\end{proof}

\begin{lem} \label{lem_10}
Let $0<\delta \le \frac 1 {100}$. Then for $i=1,2$,
\begin{align}
 \left| \int_{|\eta|\le \frac{2\delta |\xi|}{\langle \xi \rangle^2}} e^{it a_i(\xi,\eta)} d\eta
\right| \lsm \frac 1 {\langle t \rangle} \cdot \langle \xi \rangle^7. \label{e_lem10_1}
\end{align}

\end{lem}

\begin{proof}
This is a stationary phase calculation. The main point is to get the explicit dependence
on the parameter $\xi$.
 WOLOG we can assume $t \ge 1$. Define $\lambda=\frac 12 |\xi|$. By rotation invariance, we may assume
$\xi =2 \lambda \hat e_1 = (2\lambda,0)$. Let $\eta = \lambda \tilde \eta$. Then
\begin{align*}
 & \langle \eta+\frac {\xi} 2 \rangle + \langle \eta -\frac {\xi} 2 \rangle \\
= &\sqrt{ 1 +\lambda^2 (|\tilde \eta|^2 +2\tilde \eta_1+1)} +
\sqrt{ 1 +\lambda^2 (|\tilde \eta|^2 -2\tilde \eta_1+1)}.
\end{align*}
Here $\tilde \eta_1$ is the first component of the vector $\tilde \eta$. Passing to radial
coordinates, we obtain
\begin{align*}
 \text{LHS of \eqref{e_lem10_1}} \lsm
\left| \int_0^{2\pi} \int_0^{\frac {4\delta}{\sqrt{1+4\lambda^2}}} e^{itg(\rho,\theta)} \lambda^2 \rho d\rho d\theta
\right|,
\end{align*}
where
\begin{align*}
 g(\rho,\theta) = \sqrt{ 1+\lambda^2(\rho^2+2\rho \cos \theta+1)}
+\sqrt{ 1+\lambda^2(\rho^2-2\rho \cos \theta+1)}.
\end{align*}
Now set $r=\sqrt{\rho}$. Then we get
\begin{align}
 \text{LHS of \eqref{e_lem10_1}} \lsm
\left| \int_0^{2\pi} \int_0^{\frac{16\delta^2}{1+4\lambda^2}}
e^{it\tilde g(r,\theta)} \lambda^2 dr d\theta \right|.
\label{e_lem10_1a}
\end{align}
Here
\begin{align*}
 \tilde g(r,\theta) & = \sqrt{1+\lambda^2(r+2\sqrt r \cos \theta+1)} + \sqrt{1+\lambda^2(r-2\sqrt r \cos \theta +1)} \\
& =: \sqrt{A_+} + \sqrt{A_-}.
\end{align*}
Observe that for $r\le \frac {16\delta^2}{1+4\lambda^2}$,
\begin{align*}
 1\lsm A_+, \, A_- \lsm \lambda+1.
\end{align*}
On the other hand
\begin{align*}
 \frac {\partial}{\partial r} \tilde g(r,\theta) & = \lambda^2
\cdot \frac{A_+ +A_-+2 \sqrt{A_+ \cdot A_-} -4 \lambda^2 \cos^2 \theta}
{\sqrt{A_+\cdot A_-} \cdot (\sqrt{A_+}+\sqrt{A_-})} \\
& = \lambda^2 \cdot
\frac{2+2\lambda^2 (r+1) +2 \sqrt{(1+\lambda^2(r+1))^2-4\lambda^4 r\cos^2\theta} -4\lambda^2 \cos^2\theta}
{\sqrt{A_+\cdot A_-} \cdot (\sqrt{A_+} +\sqrt{A_-})}.
\end{align*}
For $r \le \frac{16\delta^2}{1+4\lambda^2}$, we have
\begin{align*}
 4\lambda^4 r \cos^2\theta \le 2 \lambda^2 \le 2\lambda^2 (r+1).
\end{align*}
This gives
\begin{align*}
 & 2+2\lambda^2(r+1) + 2 \sqrt{(1+\lambda^2(r+1))^2 -4\lambda^4 r \cos^2\theta} -4 \lambda^2 \cos^2\theta \\
\ge & 2+2\lambda^2 +2\lambda^2(r+1) -4\lambda^2 \\
\ge & 2.
\end{align*}
Therefore we get for $0\le r \le \frac{16\delta^2}{1+4\lambda^2}$,
\begin{align}
 \left| \frac {\partial}{\partial r} \tilde g(r,\theta) \right|
\gtrsim \frac {\lambda^2} {(1+\lambda)^{\frac 32}}.
\label{e_lem10_2}
\end{align}
 Next we need to bound the second order derivative $\frac {\partial^2}{\partial r^2} \tilde g(r,\theta)$.
 We first rewrite $\tilde g(r,\theta)$ as
 \begin{align*}
  \tilde g(r,\theta) & = \sqrt { ( \sqrt{A_+} + \sqrt {A_-} )^2} \\
 & = \sqrt{ 2+2\lambda^2(r+1) + 2 \sqrt{(1+\lambda^2(r+1))^2 -4\lambda^4 r \cos^2 \theta}} \\
 & =: \sqrt{B_1+ 2 \sqrt{B_2}}.
 \end{align*}
By direct differentiation, we have
\begin{align*}
 \frac{\partial^2 \tilde g}{\partial r^2} & = -\frac 1 4 \cdot (B_1+2\sqrt{B_2})^{-\frac 32}
\cdot ( \frac {\partial B_1}{\partial r} + B_2^{-\frac 12} \cdot \frac{\partial B_2}{\partial r} )^2 \\
& \quad + \frac 12 \cdot (B_1+2\sqrt{B_2})^{-\frac 12} \cdot
( -\frac 12 B_2^{-\frac 32} \cdot ( \frac{\partial B_2} {\partial r} )^2
+B_2^{-\frac 12} \cdot \frac{\partial^2 B_2}{\partial r^2} ).
\end{align*}
It is not difficult to check that for $0<r \le \frac {16 \delta^2}{1+4\lambda^2}$,
\begin{align*}
 B_1 +2\sqrt{B_2} \gtrsim 1, \quad \text{and $B_2\gtrsim 1$}, \\
\left| \frac{\partial B_1}{\partial r} \right| \lsm \lambda^2,
\quad \text{and $\left| \frac{\partial B_2} {\partial r} \right| \lsm \lambda^2+\lambda^4$}.
\end{align*}
Also
\begin{align*}
 \left| \frac {\partial^2 B_2} {\partial r^2} \right| \lsm \lambda^4.
\end{align*}
Therefore
\begin{align}
 \left| \frac{\partial^2 \tilde g } {\partial r^2} \right| &\lsm
(\lambda^2 +\lambda^4)^2 +\lambda^4 \notag \\
& \lsm \lambda^4 +\lambda^8. \label{e_lem10_3}
\end{align}
By \eqref{e_lem10_2} and \eqref{e_lem10_3}, we obtain
\begin{align*}
 \left| \int_0^{\frac {16\delta^2}{1+4\lambda^2}} e^{it\tilde g(r,\theta)} dr \right|
& \lsm \frac 1 t \cdot \left( \frac{(1+\lambda)^{\frac 32}}{\lambda^2}
+\frac{(1+\lambda)^3}{\lambda^4} \cdot (\lambda^4+\lambda^8) \cdot \frac 1{1+\lambda^2} \right) \\
& \lsm \frac 1 t \cdot \left( \frac{(1+\lambda)^{\frac 32}} {\lambda^2} +(1+\lambda)^5 \right).
\end{align*}
Therefore by \eqref{e_lem10_1a}, we get
\begin{align*}
 \text{LHS of \eqref{e_lem10_1} } \lsm \frac 1 t \cdot (1+\lambda)^7.
\end{align*}
The lemma is proved.
\end{proof}

\begin{lem} \label{lem_11}
Let $\psi \in C_c^\infty(\R^2)$ be such that
\begin{align*}
 \psi(x) = \begin{cases}
            1, \quad |x| \le 1, \\
0,\quad |x| \ge 2.
           \end{cases}
\end{align*}
Then for any $i=1,2$, we have
\begin{align}
 & \left| \int_{|\eta| \le \frac {2\delta |\xi|}{\langle \xi\rangle^2}}
\left( \frac{h(\xi,\eta)} {a_i(\xi,\eta)} \psi(\frac{\eta \langle \xi \rangle^2}{\delta |\xi|} )
-\frac{h(\xi,0)}{a_i(\xi,0)} \right) \cdot e^{it a_i(\xi,\eta)} d\eta \right| \notag \\
\lsm & \frac 1 {\langle t \rangle}
\cdot \langle \xi \rangle^9 \cdot \sup_{|\tilde \eta| \le \frac{2\delta |\xi|}{\langle \xi \rangle^2}}
( |h(\xi,\tilde \eta)| + | \nabla_\eta h(\xi, \tilde \eta)| ).
\label{e_lem_11_1}
\end{align}
\end{lem}

\begin{proof}
If $t\le 1$, then \eqref{e_lem_11_1} holds trivially. Now assume $t\ge 1$. Define
\begin{align*}
 G_i (\xi,\eta) = \frac {h(\xi,\eta)}{a_i (\xi,\eta)} \psi ( \frac{\eta \langle \xi \rangle^2}{\delta |\xi|} )
-\frac {h(\xi,0)}{a_i(\xi,0)}.
\end{align*}
Note that $\psi(0)=1$ and
\begin{align*}
 \left| \nabla_\eta ( \psi( \frac{\eta \langle \xi \rangle^2}{\delta |\xi|} ) )
\right| \lsm \frac{\langle \xi \rangle^2}{|\xi|}.
\end{align*}
By Lemma \ref{lem_3}, we have for $i=1,2$
\begin{align}
 \sup_{|\eta| \le \frac {2\delta |\xi|}{\langle \xi \rangle^2}}
|G_i(\xi,\eta)| \lsm \langle \xi \rangle \cdot \sup_{|\eta| \le \frac{2\delta |\xi|}{\langle \xi \rangle^2}}
|h(\xi, \eta)|.
\label{e_lem_11_2}
\end{align}
Since $G_i(\xi,0)=0$, we get
\begin{align}
 \sup_{|\eta| \le \frac{2\delta |\xi|}{\langle \xi \rangle^2}}
\frac{|G_i(\xi,\eta)|}{|\eta|} &
\lsm \sup_{|\eta| \le \frac{2\delta |\xi|}{\langle \xi \rangle^2}}
|\nabla_\eta G_i(\xi,\eta) |\notag \\
& \lsm \sup_{|\eta| \le \frac{2\delta |\xi|}{\langle \xi \rangle^2}}
\left( \frac{|\nabla_\eta h(\xi,\eta)|}{|a_i(\xi,\eta)|} +
\frac{|h(\xi,\eta)| \cdot |\nabla_\eta a_i(\xi,\eta)|}{|a_i(\xi,\eta)|^2} \right) \notag \\
& \qquad + \frac{\langle \xi \rangle^2}{|\xi|} \cdot \sup_{|\eta| \sim \frac{2\delta |\xi|}{\langle \xi \rangle^2}}
\frac {|h(\xi,\eta)|}{|a_i(\xi,\eta)|} \notag \\
& \lsm \frac{\langle \xi \rangle^3} {|\xi|}
\cdot \sup_{|\eta|\le \frac {2\delta |\xi|}{\langle \xi \rangle^2}}
( |h(\xi,\eta)| + | \nabla_\eta h(\xi,\eta)| ).
\label{e_lem_11_3}
\end{align}
Using integration by parts together with \eqref{e_lem_11_2}, \eqref{e_lem_11_3}, we then
obtain or $i=1,2$
\begin{align*}
 \text{LHS of \eqref{e_lem_11_1} }
& = \frac 1 t \bigl| \int_{|\eta| = \frac{2\delta |\xi|}{\langle \xi \rangle^2}}
G_i(\xi,\eta)  \frac{\nabla_\eta a_i}{|\nabla_\eta a_i|^2} \cdot \hat \eta e^{ita_i(\xi,\eta)} d \sigma(\eta) \\
& \quad - \int_{|\eta|\le \frac {2\delta |\xi|}{\langle \xi \rangle^2}}
\nabla_\eta \cdot ( G_i(\xi,\eta) \frac{\nabla_\eta a_i(\xi,\eta)}{|\nabla_\eta a_i (\xi,\eta) |^2} )
e^{ita_i(\xi, \eta)} d\eta \bigr| \\
& \lsm \frac 1 t \int_{|\eta|=\frac{2\delta |\xi|}{\langle \xi \rangle^2}}
\langle \xi \rangle \cdot \sup_{|\tilde \eta| \le \frac{2\delta |\xi|}{\langle \xi \rangle^2}}
|h(\xi,\tilde \eta)|\cdot \frac{\langle |\eta| +|\xi| \rangle^3}{|\eta|} d\sigma(\eta)  \\
& \quad + \frac 1 t \cdot \frac{\langle \xi \rangle^3}{|\xi|}
\cdot \sup_{|\tilde \eta| \le \frac{2\delta |\xi|}{\langle \xi \rangle^2}}
(|h(\xi,\tilde \eta)| + |\nabla_\eta h(\xi, \tilde \eta) | ) \cdot \int_{|\eta| \le \frac{2\delta |\xi|}{\langle \xi \rangle^2}}
\frac{\langle |\eta| +|\xi| \rangle^6}{|\eta|} d\eta \\
& \lsm \frac 1 t \cdot \langle \xi \rangle^9
\cdot \sup_{|\tilde \eta| \le \frac{2\delta |\xi|}{\langle \xi \rangle^2}}
(|h(\xi,\tilde \eta) + |\nabla_\eta h(\xi, \tilde \eta) | ).
\end{align*}

\end{proof}

\begin{lem} \label{lem_12}
 For $i=1,2$, we have
\begin{align}
 & \left| \int_{\R^2} \frac{h(\xi,\eta)}{a_i (\xi, \eta)} (1-\psi(\frac{\eta \langle \xi \rangle^2}{\delta |\xi|}))
\cdot e^{it a_i (\xi,\eta)} d\eta \right| \notag \\
\lsm & \frac 1 {\langle t \rangle} \cdot (1+|\log |\xi| |) \cdot \langle \xi \rangle^{-5}
 \cdot \sup_{\tilde \xi,\tilde \eta \in \R^2} \langle |\tilde \xi|+ |\tilde \eta| \rangle^{12}
\cdot ( |h(\tilde \xi, \tilde \eta)| + |\nabla_{\tilde \eta} h(\tilde \xi, \tilde \eta)|).
\label{e_lem_12_1}
\end{align}

\end{lem}

\begin{proof}
 WOLOG we may assume $t\ge 1$. Since the function $1-\psi(\frac{\eta \langle \xi \rangle^2}{\delta |\xi|})$ is localized
to the region $\{ |\eta| \ge \frac{\delta |\xi|}{\langle \xi \rangle^2}$ and vanishes on the boundary, a simple
integration by parts gives us
 \begin{align}
 \text{LHS of \eqref{e_lem_12_1}} & = -\frac 1 {it}
 \int_{\R^2} \nabla_\eta \cdot \left(
 \frac{\nabla_\eta a_i(\xi,\eta)}{ |\nabla_\eta a_i(\xi,\eta)|^2} \cdot
 \frac{h(\xi,\eta)}{a_i(\xi,\eta)} \cdot ( 1- \psi ( \frac{\eta \langle \xi \rangle^2}{\delta |\xi|})) \right)
 e^{it a_i(\xi,\eta)} d\eta \notag \\
 &=: -\frac 1 {it} \int_{\R^2} G_1(\xi,\eta) e^{it a_i(\xi,\eta)} d\eta. \label{e_lem_12_2}
\end{align}
It only remains for us to bound $G_1(\xi,\eta)$.

By direct differentiation, we get
\begin{align}
 |G_1(\xi,\eta)| & \lsm \frac{|\Delta_\eta a_i(\xi,\eta)|}{|\nabla_\eta a_i(\xi,\eta)|^2}
\cdot \frac {|h(\xi,\eta)|} {|a_i(\xi,\eta)|} \cdot |1-\psi( \frac{\eta \langle \xi \rangle^2}{\delta |\xi|}) |
\label{e_lem_12_2a} \\
& \quad + \frac 1 {|\nabla_\eta a_i(\xi,\eta)|}
\cdot \frac{|\nabla_\eta h(\xi,\eta)|}{|a_i(\xi,\eta)|} \cdot |1-\psi(\frac{\eta \langle \xi \rangle^2}{\delta |\xi|})|
\label{e_lem_12_3} \\
& \quad + \frac{|h(\xi,\eta)|}{|a_i(\xi,\eta)|^2} \cdot |1-\psi(\frac{\eta \langle \xi \rangle^2}{\delta |\xi|})|
\label{e_lem_12_4} \\
 & \quad +\frac 1 {|\nabla_\eta a_i(\xi,\eta)|}
 \cdot \frac{|h(\xi,\eta)|}{|a_i(\xi,\eta)|}
 \cdot \frac{\langle \xi \rangle^2}{\delta |\xi|}
 \cdot | (\nabla \psi) (\frac{\eta \langle \xi \rangle^2}{\delta |\xi|} )|
 \label{e_lem_12_5}
\end{align}

By Lemma \ref{lem_3}, we have
\begin{align*}
 \eqref{e_lem_12_2a} & \lsm \frac 1 {|\nabla_\eta a_i(\xi,\eta)|^2}
\cdot \frac{|h(\xi,\eta)|}{|a_i(\xi,\eta)|} \cdot 1_{|\eta| \ge \frac{\delta |\xi|}{\langle \xi \rangle^2}} \\
&\lsm \frac{ \langle |\xi|+|\eta| \rangle^7}{|\eta|^2}
\cdot |h(\xi,\eta)| \cdot 1_{|\eta| \ge \frac{\delta |\xi|}{\langle \xi \rangle^2}}.
\end{align*}
Similarly we compute
\begin{align*}
 \eqref{e_lem_12_3} &\lsm \frac{\langle |\xi|+|\eta| \rangle^4}{|\eta|}
\cdot |\nabla_\eta h(\xi,\eta)| \cdot 1_{|\eta| \ge \frac{\delta |\xi|}{\langle \xi \rangle^2}}, \\
\eqref{e_lem_12_4} & \lsm \langle |\xi|+|\eta| \rangle^2 \cdot |h(\xi,\eta)| \cdot 1_{|\eta| \ge \frac{\delta |\xi|}{\langle \xi \rangle^2}}, \\
\eqref{e_lem_12_5} & \lsm \frac{\langle |\xi|+|\eta| \rangle^4}{|\eta|}
\cdot \frac{\langle \xi \rangle^2}{|\xi|}
\cdot |h(\xi,\eta)| \cdot 1_{|\eta|\sim \frac{\delta |\xi|}{\langle \xi \rangle^2}}.
\end{align*}
Collecting all the estimates, we get
\begin{align*}
 |G_1(\xi,\eta)| & \lsm \frac{\langle |\xi|+|\eta| \rangle^7}{|\eta|^2}
\cdot ( |h(\xi,\eta)|+|\nabla_\eta h(\xi,\eta)|) \cdot 1_{|\eta|\ge \frac{\delta |\xi|}{\langle \xi \rangle^2}} \\
&\lsm \frac 1 {|\eta|^2} \cdot \frac 1 {\langle \eta \rangle} \cdot \langle \xi \rangle^{-5}
\cdot 1_{|\eta| \ge \frac{\delta |\xi|}{\langle \xi \rangle^2}} \\
& \qquad \cdot \sup_{\tilde \xi,\tilde \eta \in \R^2}
\langle |\tilde \xi | + |\tilde \eta| \rangle^7
\cdot \langle \tilde \xi \rangle^5 \cdot ( |h(\tilde \xi,\tilde \eta) | + |\nabla_{\tilde \eta} h(\tilde \xi,\tilde \eta)|)
\\
& \lsm \frac 1{|\eta|^2} \cdot \frac 1{\langle \eta \rangle} \cdot \langle \xi\rangle^{-5}
\cdot 1_{|\eta| \ge \frac{\delta |\xi|}{\langle \xi \rangle^2}} \\
& \qquad \cdot  \sup_{\tilde \xi,\tilde \eta \in \R^2}
\langle |\tilde \xi | + |\tilde \eta| \rangle^{12}
 \cdot ( |h(\tilde \xi,\tilde \eta) | + |\nabla_{\tilde \eta} h(\tilde \xi,\tilde \eta)|).
\end{align*}
Plugging the last estimate into \eqref{e_lem_12_2} and integrating out $\eta$, we obtain
\eqref{e_lem_12_1}.
\end{proof}

\begin{lem} \label{lem_27}
Let $i=3,4$. Let $t\ge 1$ and $|\xi| \ge \frac 1t$. Then
\begin{align}
  &\left| \int_{\R^2} \frac {h(\xi,\eta)}{ a_i(\xi,\eta)} e^{it a_i(\xi,\eta)}   d\eta \right|  \notag \\
\lsm  &\frac 1 {t^2} \cdot \frac 1 {|\xi|^2}
\cdot \frac 1 {\langle \xi \rangle^5}
 \cdot \sup_{\tilde \xi, \tilde \eta \in \R^2}
\langle |\tilde \xi| + |\tilde \eta|\rangle^{21}
( |h(\tilde \xi,\tilde \eta)| + | (\partial_{\tilde \eta} h)(\tilde \xi, \tilde \eta) |
+ |(\partial^2_{\tilde \eta} h)(\tilde \xi,\tilde \eta) |).
\label{e_tmp_340}
\end{align}

\end{lem}

\begin{proof}
 By using integration by parts twice, we get
\begin{align}
 & \int_{\R^2} \frac {h(\xi,\eta)}{ a_i(\xi,\eta)} e^{it a_i(\xi,\eta)}   d\eta \notag \\
= & -\frac 1 {t^2} \int_{\R^2}
\nabla_\eta \cdot \left( \frac {\nabla_\eta a_i(\xi,\eta)}{|\nabla_\eta a_i(\xi,\eta)|^2}
\nabla_\eta \cdot ( \frac{\nabla_\eta a_i(\xi,\eta)} {|\nabla_\eta a_i(\xi,\eta)|^2}
\frac{h(\xi,\eta)}{a_i(\xi,\eta)}) \right) e^{ita_i(\xi,\eta)} d\eta \notag \\
=: & -\frac 1 {t^2} \int_{\R^2} G_2(\xi,\eta) e^{ita_i(\xi,\eta)} d\eta. \label{e_tmp_341}
\end{align}
It only remains for us to bound $G_2(\xi,\eta)$.  By direct differentaion and discarding
inessential constants, we have
\begin{align}
 |G_2(\xi,\eta)| & \lsm \left| \partial_\eta \left( \nai \right) \right|
\cdot \left| \partial_\eta \left( \nai \right) \right|
\cdot \frac{|h(\xi,\eta)|}{\ai} \label{e_305_1} \\
& \quad+ \left| \partial_\eta \left( \nai \right) \right| \cdot \frac 1 {\gai}
\cdot \left| \partial_\eta \left( \frac{ h(\xi,\eta)} {\ai} \right) \right| \label{e_305_2} \\
& \quad + \frac 1 {\gai} \cdot \left| \partial_\eta^2 \left( \nai \right) \right| \cdot
\frac {|h(\xi,\eta)|} {\ai} \label{e_305_3} \\
& \quad + \frac 1 {\gai} \cdot \left| \partial_\eta \left( \nai \right) \right| \cdot \left|
\partial_\eta \left( \frac{ h(\xi,\eta)}{\ai} \right) \right|
\label{e_305_4} \\
& \quad +\frac 1 {\gai} \cdot \frac 1 {\gai} \cdot \left|
\partial_\eta^2 \left( \frac{h(\xi,\eta)}{\ai} \right) \right|.
\label{e_305_5}
\end{align}
For \eqref{e_305_1}, we use Lemma \ref{lem_3} to bound it as
\begin{align*}
 \eqref{e_305_1} & \lsm \frac{\langle |\xi| + |\eta| \rangle^{13} } {|\xi|^2} \cdot |h(\xi,\eta)| \\
& \lsm \frac 1 {|\xi|^2} \cdot \frac 1 {\langle \xi \rangle^5}
\cdot \frac 1 {\langle \eta \rangle^3}
\cdot \sup_{\tilde \xi, \tilde \eta \in \R^2} \langle |\tilde \xi| + |\tilde \eta|\rangle^{21}
|h(\tilde \xi,\tilde \eta) |.
\end{align*}
Similarly
\begin{align}
 \eqref{e_305_2} & \lsm \left| \partial_\eta \left( \nai \right) \right| \cdot \frac 1 {|\ai|^2}
\cdot |h(\xi,\eta)| \notag \\
& \quad + \left|\partial_\eta \left( \nai \right) \right| \cdot \frac 1 {|\ai|}
\cdot \frac 1 {\gai} \cdot | \partial_\eta h(\xi,\eta)| \notag \\
& \lsm \frac{\langle |\xi|+|\eta| \rangle^8}{|\xi|} |h(\xi,\eta)|
+ \frac{\langle |\xi|+|\eta|\rangle^6}{|\xi|}
\cdot \langle |\xi|+|\eta| \rangle \cdot \frac{\langle |\xi|+|\eta| \rangle^3}{|\xi|} \cdot |(\partial_\eta h)(\xi,\eta)|
 \notag \\
& \lsm \frac 1{|\xi|^2} \cdot \frac 1 {\langle \xi \rangle^5} \cdot \frac 1 {\langle \eta \rangle^3}
\cdot \sup_{\tilde \xi, \tilde \eta \in \R^2} \langle |\tilde \xi| + |\tilde \eta|\rangle^{21}
(|h(\tilde \xi,\tilde \eta) |+|(\partial_{\tilde \eta} h)(\tilde \xi,\tilde \eta) |).
\notag
\end{align}
The computation of \eqref{e_305_3}--\eqref{e_305_4} is similar to \eqref{e_305_2}, and we have
\begin{align*}
 \eqref{e_305_3}+\eqref{e_305_4} \lsm \frac 1{|\xi|^2} \cdot \frac 1 {\langle \xi \rangle^5} \cdot \frac 1 {\langle \eta \rangle^3}
\cdot \sup_{\tilde \xi, \tilde \eta \in \R^2} \langle |\tilde \xi| + |\tilde \eta|\rangle^{21}
(|h(\tilde \xi,\tilde \eta) |+|(\partial_{\tilde \eta} h)(\tilde \xi,\tilde \eta) |).
\end{align*}
Finally for \eqref{e_305_5}, we note that by Lemma \ref{lem_3},
\begin{align*}
 \left| \partial_\eta^2 \left( \frac{h(\xi,\eta)}{\ai} \right) \right|
& \lsm | (\partial_\eta^2 h)(\xi,\eta)| +\frac{ |(\partial_\eta h)(\xi,\eta)| \cdot |\partial_\eta a_i(\xi,\eta)|}
{a_i(\xi,\eta)^2}  \\
& \quad + |h(\xi,\eta)| \cdot ( \frac{|\partial_\eta^2 a_i(\xi,\eta)|}{a_i(\xi,\eta)^2} +
\frac{|\partial_\eta a_i(\xi,\eta)|^2} {a_i(\xi,\eta)^3} ) \\
& \lsm \langle |\xi|+|\eta| \rangle^3 (|h(\xi,\eta)|+|(\partial_\eta h)(\xi,\eta)|+|(\partial_\eta^2 h)(\xi,\eta)|).
\end{align*}
Therefore
\begin{align*}
 \eqref{e_305_5} & \lsm  \frac{\langle |\xi|+|\eta| \rangle^9} {|\xi|^2} \cdot
(|h(\xi,\eta)|+|(\partial_\eta h)(\xi,\eta)|+|(\partial_\eta^2 h)(\xi,\eta)|) \\
& \lsm \frac 1{|\xi|^2} \cdot \frac 1 {\langle \xi \rangle^5} \cdot \frac 1 {\langle \eta \rangle^3}
\cdot \sup_{\tilde \xi, \tilde \eta \in \R^2} \langle |\tilde \xi| + |\tilde \eta|\rangle^{21}
(|h(\tilde \xi,\tilde \eta) |+|(\partial_{\tilde \eta} h)(\tilde \xi,\tilde \eta) |+
|(\partial_{\tilde \eta}^2 h)(\tilde \xi,\tilde \eta)|).
\end{align*}
Collecting all the estimates, we obtain
\begin{align*}
|G_2(\xi,\eta)| \lsm \frac 1{|\xi|^2} \cdot \frac 1 {\langle \xi \rangle^5} \cdot \frac 1 {\langle \eta \rangle^3}
\cdot \sup_{\tilde \xi, \tilde \eta \in \R^2} \langle |\tilde \xi| + |\tilde \eta|\rangle^{21}
(|h(\tilde \xi,\tilde \eta) |+|(\partial_{\tilde \eta} h)(\tilde \xi,\tilde \eta) |+
|(\partial_{\tilde \eta}^2 h)(\tilde \xi,\tilde \eta)|).
\end{align*}
Plugging this last estimate into RHS of \eqref{e_tmp_341} and integrating out $\eta$, we obtain
\eqref{e_tmp_340}. The lemma is proved.
\end{proof}
We are now ready to complete the
\begin{proof}[Proof of Proposition \ref{prop_51}]
Assume $h=h(\xi,\eta)$ is such that the RHS of \eqref{e_tm_128} is finite. By Lemma \ref{lem_102},
we have for all $1\le i\le 4$,
\begin{align*}
 & \left \| \langle \xi \rangle^2 \langle t \rangle \int_t^\infty
\int_{\R^2} e^{is a_i(\xi,\eta)} h(\xi,\eta) d\eta ds \right\|_{L_\xi^2} \\
= & \left \| \langle \xi \rangle^2 \langle t \rangle
\int_{\R^2} \frac {e^{it a_i(\xi,\eta)}} {a_i(\xi,\eta)} h(\xi,\eta) d\eta \right\|_{L_\xi^2}.
\end{align*}
Consider first $i=1,2$. Let $\psi \in C_c^\infty(\R^2)$ be such that
\begin{align*}
 \psi(x) =
\begin{cases}
 1,\quad |x|\le 1,\\
0,\quad |x| \ge 2.
\end{cases}
\end{align*}
Let $0<\delta<\frac 1{100}$. Then we decompose
\begin{align}
 \int_{\R^2} \frac{e^{ita_i(\xi,\eta)}} {a_i(\xi,\eta)} h(\xi,\eta) d\eta &
= \int_{\R^2} \frac{h(\xi,\eta)}{a_i(\xi,\eta)} (1-\psi(\frac{\eta \langle \xi \rangle^2}{\delta |\xi|} ))
e^{it a_i(\xi,\eta)} d\eta \label{e_51_a} \\
& \quad+ \int_{|\eta| \le \frac{2\delta |\xi|}{\langle \xi \rangle^2}}
\left( \frac{h(\xi,\eta)}{a_i(\xi,\eta)} \psi(\frac{\eta \langle \xi \rangle^2}{\delta |\xi|} )
-\frac{h(\xi,0)}{a_i(\xi,0)} \right) e^{it a_i(\xi,\eta)} d\eta \label{e_51_b} \\
& \quad + \int_{|\eta| \le \frac{2\delta |\xi|}{\langle \xi \rangle^2}}
e^{it a_i(\xi,\eta)} d\eta \cdot \frac{h(\xi,0)}{a_i(\xi,0)}. \label{e_51_c}
\end{align}
For \eqref{e_51_a}, we use Lemma \ref{lem_12} and this gives us
\begin{align*}
&\left\| \langle \xi \rangle^2 \langle t \rangle
\int_{\R^2} \frac{h(\xi,\eta)}{a_i(\xi,\eta)} (1-\psi(\frac{\eta \langle \xi \rangle^2}{\delta |\xi|} ))
e^{it a_i(\xi,\eta)} d\eta  \right\|_{L_\xi^2 } \\
\lsm & \sup_{\tilde \xi, \tilde \eta \in \R^2} \langle |\tilde \xi| + |\tilde \eta| \rangle^{12}
\cdot ( |h(\tilde \xi,\tilde \eta)| + |\nabla_{\tilde \eta} h(\tilde \xi,\tilde \eta) |)
\cdot \| (1+|\log |\xi|) \cdot \langle \xi \rangle^{-3} \|_{L_\xi^2} \\
\lsm & \sup_{\tilde \xi, \tilde \eta \in \R^2} \langle |\tilde \xi| + |\tilde \eta| \rangle^{12}
\cdot ( |h(\tilde \xi,\tilde \eta)| + |\nabla_{\tilde \eta} h(\tilde \xi,\tilde \eta) |).
\end{align*}
Similarly for \eqref{e_51_b}, we use Lemma \ref{lem_11} to compute
\begin{align*}
&\left\| \langle \xi \rangle^2 \langle t \rangle
 \int_{|\eta| \le \frac{2\delta |\xi|}{\langle \xi \rangle^2}}
\left( \frac{h(\xi,\eta)}{a_i(\xi,\eta)} \psi(\frac{\eta \langle \xi \rangle^2}{\delta |\xi|} )
-\frac{h(\xi,0)}{a_i(\xi,0)} \right) e^{it a_i(\xi,\eta)} d\eta
\right\|_{L_\xi^2 } \\
\lsm & \left\| \langle \xi \rangle^{11} \sup_{|\tilde \eta| \le \frac{2\delta |\xi|}{\langle \xi \rangle^2}}
(|h(\xi,\tilde \eta)| + |\nabla_{\tilde \eta} h(\xi,\tilde \eta)|) \right \|_{L_\xi^2} \\
\lsm &
\sup_{\tilde \xi, \tilde \eta \in \R^2} \langle |\tilde \xi| + |\tilde \eta| \rangle^{14}
\cdot ( |h(\tilde \xi,\tilde \eta)| + |\nabla_{\tilde \eta} h(\tilde \xi,\tilde \eta) |) \cdot
\| \langle \xi \rangle^{-3} \|_{L_\xi^2} \\
\lsm & \sup_{\tilde \xi, \tilde \eta \in \R^2} \langle |\tilde \xi| + |\tilde \eta| \rangle^{14}
\cdot ( |h(\tilde \xi,\tilde \eta)| + |\nabla_{\tilde \eta} h(\tilde \xi,\tilde \eta) |).
\end{align*}
For \eqref{e_51_c}, we apply Lemma \ref{lem_10}, Lemma \ref{lem_3} and obtain
\begin{align*}
 &\left\| \langle \xi \rangle^2 \langle t \rangle
\int_{|\eta| \le \frac{2\delta |\xi|}{\langle \xi \rangle^2}}
e^{it a_i(\xi,\eta)} d\eta \cdot \frac{h(\xi,0)}{a_i(\xi,0)}
\right\|_{L_\xi^2 } \\
\lsm & \left\| \langle \xi \rangle^9 \cdot \frac{h(\xi,0)}{a_i(\xi,0)} \right \|_{L_\xi^2} \\
\lsm & \| \langle \xi \rangle^{10} h(\xi,0) \|_{L_\xi^2} \\
\lsm & \sup_{\tilde \xi, \tilde \eta \in \R^2} \langle |\tilde \xi| + |\tilde \eta| \rangle^{14}
\cdot ( |h(\tilde \xi,\tilde \eta)| + |\nabla_{\tilde \eta} h(\tilde \xi,\tilde \eta) |).
\end{align*}
Therefore we have shown that in the case $i=1,2$, the contribution of $a_i$ is bounded by the RHS
of \eqref{e_tm_128}. Next we consider $i=3,4$. If $t\le 1$, then by Lemma \ref{lem_3}, we have
\begin{align*}
 &\left \| \langle \xi \rangle^2 \langle t \rangle \int_{\R^2} \frac{h(\xi,\eta)}{a_i(\xi,\eta)}
e^{it a_i(\xi,\eta)} d\eta \right\|_{L_\xi^2} \\
\lsm & \left \| \langle \xi \rangle^2 \int_{\R^2} \langle |\xi|+|\eta| \rangle |h(\xi,\eta)| d\eta
\right\|_{L_\xi^2} \\
\lsm & \sup_{\tilde \xi, \tilde \eta \in \R^2} \langle |\tilde \xi| + |\tilde \eta| \rangle^{12}
|h(\tilde \xi,\tilde \eta)|
\end{align*}
which is acceptable. If $t\ge 1$ and $|\xi| \le \frac 1t$, then we also have
\begin{align*}
 &\left \| \langle \xi \rangle^2 \langle t \rangle \int_{\R^2} \frac{h(\xi,\eta)}{a_i(\xi,\eta)}
e^{it a_i(\xi,\eta)} d\eta \right\|_{L_\xi^2(|\xi|\le \frac 1 t)} \\
\lsm & \left \|  \int_{\R^2} \langle |\xi|+|\eta| \rangle |h(\xi,\eta)| d\eta
\right\|_{L_\xi^\infty} \\
\lsm & \sup_{\tilde \xi, \tilde \eta \in \R^2} \langle |\tilde \xi| + |\tilde \eta| \rangle^{12}
|h(\tilde \xi,\tilde \eta)|.
\end{align*}
It remains for us to consider the case $t\ge 1$ and $|\xi| \ge \frac 1 t$. By Lemma \ref{lem_27},
we compute
\begin{align*}
 &\left \| \langle \xi \rangle^2 \langle t \rangle \int_{\R^2} \frac{h(\xi,\eta)}{a_i(\xi,\eta)}
e^{it a_i(\xi,\eta)} d\eta \right\|_{L_\xi^2} \\
\lsm &  \sup_{\tilde \xi, \tilde \eta \in \R^2}
\langle |\tilde \xi| + |\tilde \eta|\rangle^{21}
( |h(\tilde \xi,\tilde \eta)| + | (\partial_{\tilde \eta} h)(\tilde \xi, \tilde \eta) |
+ |(\partial^2_{\tilde \eta} h)(\tilde \xi,\tilde \eta) |) \\
& \quad \cdot \frac 1 t \cdot
\left\| \frac 1 {|\xi|^2} \cdot \frac 1 {\langle \xi \rangle^3} \right \|_{L_\xi^2(|\xi|\ge \frac 1t)} \\
\lsm & \sup_{\tilde \xi, \tilde \eta \in \R^2}
\langle |\tilde \xi| + |\tilde \eta|\rangle^{21}
( |h(\tilde \xi,\tilde \eta)| + | (\partial_{\tilde \eta} h)(\tilde \xi, \tilde \eta) |
+ |(\partial^2_{\tilde \eta} h)(\tilde \xi,\tilde \eta) |).
\end{align*}
This finishes the case $i=3,4$.
\end{proof}


\section{Solvability of the iteration system \eqref{eq10a}--\eqref{eq10b}} \label{sec_solv}
Since the RHS of \eqref{eq10a}--\eqref{eq10b} generally contains second order
derivatives of $(\bar u^{(n)}, \bar \vv^{(n)})$, the solvability of the system
is not immediately a trivial matter. However, due to the time cutoff function
$\chi_n(t)$ which vanishes for $t\ge 2^{n+1}$, these troublesome terms involving
$(\bar u^{(n)}, \bar \vv^{(n)})$ also drops out for $t\ge 2^{n+1}$. Our strategy
then is to solve first \eqref{eq10a}--\eqref{eq10b} for $t\ge 2^{n+1}$. In this regime
we can make use of the explicit representation of the solution in terms of the RHS of
\eqref{eq10a}--\eqref{eq10b} which only involve $(\Phi_1,\Phi_2,h_1,h_2,\bar u^{(n-1)}, \bar \vv^{(n-1)})$.
We then solve \emph{backwards} the same system for $t\le 2^{n+1}$ with initial
data $(\bar u^{(n)}(2^{n+1}), \bar \vv^{(n)} (2^{n+1}))$. In the latter case, the time
interval is of finite length and the solvability is rather obvious since it is a linear equation
for $(\bar u^{(n)}, \bar \vv^{(n)})$.

More precisely, we inductively assume that
\begin{align}
\| \bar u^{(n-1)} (t) \|_{H^m} + \| (\bar u^{(n-1)} (t) )^\prime \|_{H^m}
+ \| \bar \vv^{(n-1)}(t) \|_{H^m} \notag \\
 + \| ( \bar \vv^{(n-1)} (t) )^\prime \|_{H^m}
\lsm \frac {\epsilon}{\langle t \rangle}. \label{e33}
\end{align}

Then for $t\ge 2^{n+1}$, we get

\begin{align}
(\square +1) \bar u^{(n)} & = \nabla \cdot \Bigl( (\vv^{(n-1)} \cdot \nabla) \Phi_2 + u^{(n-1)} \nabla \Phi_1 \Bigr) \notag \\
& \quad - \partial_t \Bigl( (\vv^{(n-1)} \cdot \nabla) \Phi_1 + u^{(n-1)} \nabla \cdot \Phi_2 \Bigr) \notag \\
& \quad + \nabla \cdot \Bigl(
\bigl( (\bar \vv^{(n-1)} + \Phi_2) \cdot \nabla \bigr) h_2 + (\bar u^{(n-1)} + \Phi_1) \nabla h_1 \Bigr) \notag \\
& \quad - \partial_t
\Bigl( \bigl( (\bar v^{(n-1)} + \Phi_2 ) \cdot \nabla \bigr) h_1 + ( \bar u^{(n-1)} + \Phi_1 ) \nabla \cdot h_2 \Bigr), \label{e34}
\end{align}

and
\begin{align}
(\square +1) \bar \vv^{(n)} & =\nabla \Bigl( (\vv^{(n-1)} \cdot \nabla) \Phi_1 + u^{(n-1)} \nabla \cdot \Phi_2 \Bigr) \notag \\
& \quad -\partial_t \Bigl( (\vv^{(n-1)} \cdot \nabla ) \Phi_2 + u^{(n-1)} \nabla \Phi_1 \Bigr) \notag \\
& \quad + \nabla \Bigl( (\bar \vv^{(n-1)} + \Phi_2) \cdot \nabla h_1 + ( \bar u^{(n-1)} + \Phi_1) \nabla \cdot h_2 \Bigr) \notag \\
& \quad - \partial_t \Bigl(
\bigl( (\bar \vv^{(n-1)} + \Phi_2) \cdot \nabla \bigr) h_2 + (\bar u^{(n-1)} + \Phi_1) \nabla h_1 \Bigr) \notag \\
& \qquad - \nabla \Delta^{-1} \nabla \cdot( u^{(n-1)} \vv^{(n-1)} ). \label{eq344}
\end{align}

Observe that the RHS of \eqref{e34} does not contain $\bar u^{(n)}$,
therefore we have the explicit representation
\begin{align}
\bar u^{(n)}(t) = \int_\infty^t \frac {\sin (t-s)\jnab} {\jnab} \tilde G_1(\Phi_1,\Phi_2,h_1,h_2,\bar u^{(n-1)}, & \bar \vv^{(n-1)} )(s) ds,
\label{e40} \\
& \quad \text{for $t\ge 2^{n+1}$}, \notag
\end{align}
where
\begin{align}
&\quad \tilde G_1(\Phi_1,\Phi_2,h_1,h_2,\bar u^{(n-1)},  \bar \vv^{(n-1)} ) \notag \\
& = \nabla \cdot \Bigl( (\vv^{(n-1)} \cdot \nabla) \Phi_2 + u^{(n-1)} \nabla \Phi_1 \Bigr) \label{e35a} \\
& \quad - \partial_t \Bigl( (\vv^{(n-1)} \cdot \nabla) \Phi_1 + u^{(n-1)} \nabla \cdot \Phi_2 \Bigr) \label{e35b} \\
& \quad + \nabla \cdot \Bigl(
\bigl( (\bar \vv^{(n-1)} + \Phi_2) \cdot \nabla \bigr) h_2 + (\bar u^{(n-1)} + \Phi_1) \nabla h_1 \Bigr) \label{e35c} \\
& \quad - \partial_t
\Bigl( \bigl( (\bar v^{(n-1)} + \Phi_2 ) \cdot \nabla \bigr) h_1 + ( \bar u^{(n-1)} + \Phi_1 ) \nabla \cdot h_2 \Bigr). \label{e35d}
\end{align}

Similarly for $\bar \vv^{(n)}(t)$ we have
\begin{align}
\bar \vv^{(n)}(t) = \int_\infty^t \frac {\sin (t-s)\jnab} {\jnab} \tilde G_2(\Phi_1,\Phi_2,h_1,h_2,\bar u^{(n-1)}, & \bar \vv^{(n-1)} )(s) ds,
\label{e41} \\
& \quad \text{for $t\ge 2^{n+1}$}, \notag
\end{align}
where
\begin{align}
&\quad \tilde G_2(\Phi_1,\Phi_2,h_1,h_2,\bar u^{(n-1)},  \bar \vv^{(n-1)} ) \notag \\
&= \nabla \Bigl( (\vv^{(n-1)} \cdot \nabla) \Phi_1 + u^{(n-1)} \nabla \cdot \Phi_2 \Bigr) \notag \\
& \quad -\partial_t \Bigl( (\vv^{(n-1)} \cdot \nabla ) \Phi_2 + u^{(n-1)} \nabla \Phi_1 \Bigr) \notag \\
& \quad + \nabla \Bigl( (\bar \vv^{(n-1)} + \Phi_2) \cdot \nabla h_1 + ( \bar u^{(n-1)} + \Phi_1) \nabla \cdot h_2 \Bigr) \notag \\
& \quad - \partial_t \Bigl(
\bigl( (\bar \vv^{(n-1)} + \Phi_2) \cdot \nabla \bigr) h_2 + (\bar u^{(n-1)} + \Phi_1) \nabla h_1 \Bigr) \notag \\
& \qquad - \nabla \Delta^{-1} \nabla \cdot( u^{(n-1)} \vv^{(n-1)} ).
\end{align}

For the function $(\Phi_1,\Phi_2)$, we have
\begin{align}
\| \Phi_1 (s) \|_{H^{m+2}} + \|\Phi_2(s) \|_{H^{m+2}} + \| ( \Phi_1 (s) )^\prime \|_{H^{m+1}} \notag \\
+\| ( \Phi_2(s) )^\prime \|_{H^{m+1}} \lsm \frac {\epsilon} {\langle s \rangle},\quad \forall\, s\ge 0. \label{e33a}
\end{align}

Also for $(h_1,h_2)$, we have the estimate
\begin{align}
\|h_1(s) \|_{L_x^\infty} + \| \partial_s h_1(s) \|_{L_x^\infty} + \| D^{m+2} h_1(s) \|_{L_x^\infty} \notag \\
+ \| \partial_s D^{m+1} h_1(s) \|_{L_x^\infty} \lsm \frac {\epsilon}{\langle s \rangle}, \quad \forall\, s\ge 0. \label{e33b} \\
\|h_2(s) \|_{L_x^\infty} + \| \partial_s h_2(s) \|_{L_x^\infty} + \| D^{m+2} h_2(s) \|_{L_x^\infty} \notag \\
+ \| \partial_s D^{m+1} h_2(s) \|_{L_x^\infty} \lsm \frac {\epsilon}{\langle s \rangle}, \quad \forall\, s\ge 0. \label{e33c}
\end{align}

Recall that
\begin{align*}
u^{(n-1)} &= \bar u^{(n-1)} + \Phi_1 +h_1, \\
\vv^{(n-1)} & = \bar \vv^{(n-1)} + \Phi_2 +h_2.
\end{align*}

Therefore by \eqref{e33}, \eqref{e33a}--\eqref{e33c}, we have
\begin{align*}
& \| \nabla \cdot  \Bigl ( (\vv^{(n-1)} (s) \cdot \nabla ) \Phi_2(s) + u^{(n-1)}(s) \nabla \Phi_1(s) \Bigr) \|_{H^m} \\
\lsm & \| \nabla \cdot \Bigl( (\bar \vv^{(n-1)} (s) \cdot \nabla )\Phi_2(s) \Bigr) \|_{H^m} + \| \nabla \cdot ( \bar u^{(n-1)}(s) \nabla \Phi_1(s) ) \|_{H^m} \\
& \quad + \| \nabla \cdot \Bigl( (h_2(s) \cdot \nabla)  \Phi_2(s) \Bigr) \|_{H^m} + \| \nabla \cdot (h_1(s) \nabla \Phi_1(s) ) \|_{H^m} \\
& \quad + \| \nabla \cdot \Bigl( (\Phi_2(s) \cdot \nabla)\Phi_2(s) \Bigr) \|_{H^m} + \| \nabla \cdot ( \Phi_1(s) \nabla \Phi_1(s) ) \|_{H^m} \\
\lsm & \| ( \bar \vv^{(n-1)} (s) )^\prime \|_{H^m} \cdot \| \Phi_2(s) \|_{H^{m+2}} + \| ( \bar u^{(n-1)}(s))^\prime \|_{H^m} \cdot \| \Phi_1(s) \|_{H^{m+2}} \\
& \quad +\Bigl( \| h_2(s) \|_{L_x^\infty} + \| D^{m+1} h_2(s) \|_{L_x^\infty} \Bigr) \cdot \| \Phi_2(s) \|_{H^{m+2}} \\
& \quad + \Bigl( \|h_1(s) \|_{L_x^\infty} + \|D^{m+1} h_1(s) \|_{L_x^\infty} \Bigr) \cdot \|\Phi_1(s) \|_{H^{m+2}} \\
& \qquad + \| \Phi_2(s) \|^2_{H^{m+2}} + \| \Phi_1(s) \|^2_{H^{m+2}} \\
\lsm & \; \frac {\epsilon^2} {\langle s \rangle^2}.
\end{align*}

The remaining terms \eqref{e35b}--\eqref{e35d} can be estimated in a similar fashion. Collecting the estimates, we get
\begin{align*}
\|\tilde G_1(\Phi_1,\Phi_2,h_1,h_2,\bar u^{(n-1)},  \bar \vv^{(n-1)} )(s) \|_{H^m} \lsm \frac {\epsilon^2} {\langle s \rangle^2}.
\end{align*}

By \eqref{e40}, we then compute
\begin{align*}
\| \partial_t \bar u^{(n)} (t) \|_{H^m} & \lsm \int_t^\infty \| \tilde G_1(\Phi_1,\Phi_2,h_1,h_2,\bar u^{(n-1)},  \bar \vv^{(n-1)} )  (s) \|_{H^m} ds \\
& \lsm \int_t^\infty \frac {\epsilon^2} {\langle s \rangle^2} ds \\
& \lsm \frac {\epsilon^2} {\langle t \rangle}, \quad \forall\, t\ge 2^{n+1}.
\end{align*}

Similarly
\begin{align*}
\| \bar u^{(n)} (t) \|_{H^m} + \| \partial_{x_i} \bar u^{(n)} (t) \|_{H^m} \lsm \frac {\epsilon^2} {\langle t \rangle}, \quad\forall\, t\ge 2^{n+1}.
\end{align*}

Therefore
\begin{align}
\| \bar u^{(n)} (t) \|_{H^m} + \| (\bar u^{(n)} (t) )^\prime \|_{H^m} \lsm \frac {\epsilon^2} {\langle t \rangle}, \quad\forall\, t\ge 2^{n+1}.
\label{e47a}
\end{align}

By using similar estimates, we obtain from \eqref{e41}
\begin{align}
\| \bar \vv^{(n)} (t) \|_{H^m} + \| (\bar \vv^{(n)} (t) )^\prime \|_{H^m} \lsm \frac {\epsilon^2} {\langle t \rangle}, \quad \forall\, t\ge 2^{n+1}.
\label{e47b}
\end{align}

The estimates \eqref{e47a}--\eqref{e47b} are consistent with the inductive assumption \eqref{e33} on the time interval $[2^{n+1},\infty)$.
However we still have to solve \eqref{eq10a}--\eqref{eq10b} for $0\le t \le 2^{n+1}$. Since we have obtained the solution on
the time interval $[2^{n+1},\infty)$, we can use $(\bar u^{(n)}(2^{n+1}), \bar \vv^{(n)}(2^{n+1}))$ as initial data which is explicitly
given by \eqref{e40}--\eqref{e41}, i.e.
\begin{align*}
\bar u^{(n)}(2^{n+1}) &= \int_\infty^{2^{n+1}} \frac {\sin (2^{n+1}-s)\jnab} {\jnab} \tilde G_1(\Phi_1,\Phi_2,h_1,h_2,\bar u^{(n-1)},  \bar \vv^{(n-1)} )(s) ds,\\
\bar \vv^{(n)}(2^{n+1}) &= \int_\infty^{2^{n+1}} \frac {\sin (2^{n+1}-s)\jnab} {\jnab} \tilde G_2(\Phi_1,\Phi_2,h_1,h_2,\bar u^{(n-1)},  \bar \vv^{(n-1)} )(s) ds.
\end{align*}

Using this information, we then
solve \emph{backwards} \eqref{eq10a}--\eqref{eq10b} until $t=0$. This way we have obtained the solution $(\bar u^{(n)}(t), \bar \vv^{(n)}(t))$
well defined
for all $t\ge 0$. Furthermore since $[0,2^{n+1}]$ is a finite time interval and $(\bar u^{(n)}(t), \bar \vv^{(n)}(t))$ enjoys the decay estimate
\eqref{e47a}--\eqref{e47b}, we have
\begin{align}
\| \bar u^{(n)} (t) \|_{H^m} &+ \| \bar \vv^{(n)}(t) \|_{H^m} + \| ( \bar u^{(n)}(t) )^\prime \|_{H^m} \notag \\
& \quad \| (\bar \vv^{(n)} (t) )^\prime \|_{H^m} \lsm \frac {K_n} {\langle t \rangle}, \quad \forall\, t\ge 0, \label{eApriori}
\end{align}
where $K_n$ is a constant depending on the iteration step $n$. Because of the dependence on $n$, the estimate \eqref{eApriori} is certainly not
good for us since we have to justify the inductive assumption \eqref{e33} for step $n$. Nevertheless it is an important a priori estimate for our
argument later. In the next section we shall upgrade the estimate \eqref{eApriori} to the
true decay estimate \eqref{e33}.

\section{A priori energy estimates of the iteration system \eqref{eq10a}--\eqref{eq10b}} \label{sec_uni}

In this section we carry out basic energy estimates for the iteration system \eqref{eq10a}--\eqref{eq10b}. Our standing inductive assumption is
\begin{align}
\| \bar u^{(n-1)} (t) \|_{H^m} + \| (\bar u^{(n-1)} (t) )^\prime \|_{H^m}
+ \| \bar \vv^{(n-1)}(t) \|_{H^m} \notag \\
 + \| ( \bar \vv^{(n-1)} (t) )^\prime \|_{H^m}
\lsm \frac {\epsilon}{\langle t \rangle}. \label{e100}
\end{align}

For the functions $(\Phi_1,\Phi_2)$, recall the estimate
\begin{align}
\| \Phi_1 (t) \|_{H^{m+2}} + \|\Phi_2(t) \|_{H^{m+2}} + \| ( \Phi_1 (t) )^\prime \|_{H^{m+1}} \notag \\
+\| ( \Phi_2(t) )^\prime \|_{H^{m+1}} \lsm \frac {\epsilon} {\langle t \rangle},\quad \forall\, t\ge 0. \label{e100a}
\end{align}

Also for $(h_1,h_2)$, we shall use the estimate
\begin{align}
\|h_1(t) \|_{L_x^\infty} + \| \partial_t h_1(t) \|_{L_x^\infty} + \| D^{m+2} h_1(t) \|_{L_x^\infty} \notag \\
+ \| \partial_t D^{m+1} h_1(t) \|_{L_x^\infty} \lsm \frac {\epsilon}{\langle t \rangle}, \quad \forall\, t\ge 0. \label{e100b} \\
\|h_2(t) \|_{L_x^\infty} + \| \partial_t h_2(t) \|_{L_x^\infty} + \| D^{m+2} h_2(t) \|_{L_x^\infty} \notag \\
+ \| \partial_t D^{m+1} h_2(t) \|_{L_x^\infty} \lsm \frac {\epsilon}{\langle t \rangle}, \quad \forall\, t\ge 0. \label{e100c}
\end{align}

Our task is to justify the estimate \eqref{e100} for $(\bar u^{(n)}, \bar \vv^{(n)})$. For the convenience of notations, we introduce
\begin{align*}
A_n(t) &= \| \bar u^{(n)} (t) \|_{H^m} + \| (\bar u^{(n)} (t) )^\prime \|_{H^m}
+ \| \bar \vv^{(n)}(t) \|_{H^m}  \\
 &\qquad+ \| ( \bar \vv^{(n)} (t) )^\prime \|_{H^m}, \quad \forall\, n\ge 0.
\end{align*}

Take $D^\alpha$ derivative on both sides of \eqref{eq10a} and multiply by the factor $\partial_t D^\alpha \bar u^{(n)}$. Integrating by
parts and summing over $|\alpha| \le m$, we obtain
\begin{align}
 & \frac 12 \frac d {dt} \Bigl( \sum_{|\alpha| \le m} \| D^\alpha \bar u^{(n)} (t) \|^2_{L_x^2} + \| \left( D^\alpha \bar u^{(n)} (t) \right)^\prime
 \|_{L_x^2}^2 \Bigr) \notag \\
 = & \chi_n(t) \sum_{|\alpha| \le m} \int_{\mathbb R^2} D^\alpha \left( \nabla \cdot
\Bigl( (\vv^{(n-1)} \cdot \nabla)(\bar \vv^{(n)}) + u^{(n-1)} \nabla \bar u^{(n)} \Bigr) \right)
\partial_t D^\alpha \bar u^{(n)} dx \label{52a} \\
& \quad - \chi_n(t) \sum_{|\alpha| \le m} \int_{\mathbb R^2} D^\alpha
\partial_t \Bigl( (\vv^{(n-1)} \cdot \nabla) (\bar u^{(n)} ) + u^{(n-1)} \nabla \cdot ( \bar \vv^{(n)}) \Bigr)
\partial_t D^\alpha \bar u^{(n)} dx \label{52b} \\
& \quad + \sum_{|\alpha| \le m} \int_{\mathbb R^2} D^\alpha \nabla \cdot \Bigl( (\vv^{(n-1)} \cdot \nabla) \Phi_2 + u^{(n-1)} \nabla \Phi_1 \Bigr)
\partial_t D^\alpha \bar u^{(n)} dx \label{52c} \\
& \quad - \sum_{|\alpha| \le m} \int_{\mathbb R^2} D^\alpha
\partial_t \Bigl( (\vv^{(n-1)} \cdot \nabla) \Phi_1 + u^{(n-1)} \nabla \cdot \Phi_2 \Bigr)
\partial_t D^\alpha \bar u^{(n)} dx \label{52d} \\
& \quad + \sum_{|\alpha| \le m} \int_{\mathbb R^2} D^\alpha \nabla \cdot \Bigl(
\bigl( (\bar \vv^{(n-1)} + \Phi_2) \cdot \nabla \bigr) h_2 + (\bar u^{(n-1)} + \Phi_1) \nabla h_1 \Bigr)
\partial_t D^\alpha \bar u^{(n)} dx \label{52e} \\
& \quad - \sum_{|\alpha| \le m} \int_{\mathbb R^2} D^\alpha \partial_t
\Bigl( \bigl( (\bar \vv^{(n-1)} + \Phi_2 ) \cdot \nabla \bigr) h_1 + ( \bar u^{(n-1)} + \Phi_1 ) \nabla \cdot h_2 \Bigr)
\partial_t D^\alpha \bar u^{(n)} dx. \label{52f}
\end{align}

We first deal with \eqref{52a}. By using integration by parts and Leibniz in time, we write
\begin{align}
\eqref{52a} & = \chi_n(t) \sum_{|\alpha| \le m-1} \int_{\mathbb R^2} D^\alpha \nabla \cdot \Bigl(
(\vv^{(n-1)} \cdot \nabla) (\bar \vv^{(n)} ) + u^{(n-1)} \nabla \bar u^{(n)} \Bigr) \partial_t D^\alpha \bar u^{(n)} dx
\label{52aa} \\
& \quad + \frac d {dt} \Bigl( \chi_n(t) \sum_{\substack{|\alpha|=m \\ 1\le i,j\le 2 }}
\int_{\mathbb R^2} \vv_j^{(n-1)} D^\alpha \partial_{x_j} \partial_{x_i} \bar \vv_i^{(n)} D^\alpha \bar u^{(n)} dx \Bigr)
\label{52ab} \\
& \quad - \chi_n^\prime (t) \sum_{\substack{|\alpha|=m \\ 1\le i,j\le 2}} \int_{\R^2} \vv_j^{(n-1)} D^\alpha \partial_{x_j}
\partial_{x_i} \bar v_i^{(n)} D^\alpha \bar u^{(n)} dx \label{52ac} \\
& \quad - \chi_n(t) \sum_{\substack{|\alpha|=m \\ 1\le i,j\le 2}} \int_{\R^2} \partial_t \vv_j^{(n-1)} D^\alpha
\partial_{x_j} \partial_{x_i} \bar \vv_{i}^{(n)} D^\alpha \bar u^{(n)} dx \label{52ad} \\
& \quad -\chi_n(t) \sum_{\substack{|\alpha|=m \\ 1\le i,j\le 2}} \int_{\R^2} \vv_j^{(n-1)} \partial_t D^\alpha \partial_{x_j} \partial_{x_i}
\bar \vv_i^{(n)} D^\alpha \bar u^{(n)} dx. \label{52ae}
\end{align}

By using \eqref{e100}, \eqref{e100a}--\eqref{e100c} and Sobolev embedding, we compute the absolute value of
\eqref{52aa} as follows
\begin{align*}
|\eqref{52aa} | & \lsm \sum_{|\alpha| \le m-1} \sum_{|\beta| \le |\alpha|+1} \int_{\R^2} | D^\beta \vv^{(n-1)} |
\cdot | D^{\alpha+2-\beta} \bar \vv^{(n)} | \cdot |\partial_t D^\alpha \bar u^{(n)} | dx \\
& \quad+ \int_{\R^2} |D^\beta u^{(n-1)} | \cdot |D^{\alpha+2-\beta} \bar u^{(n)} | \cdot |\partial_t D^\alpha
\bar u^{(n)} | dx \\
& \lsm \sum_{|\alpha| \le m-1} \sum_{|\beta| \le |\alpha|+1} \| D^\beta h_2(t) \|_{L_x^\infty} \cdot \|D^{\alpha+2-\beta} \bar \vv^{(n)} (t) \|_{L_x^2}
\cdot \| \partial_t D^\alpha \bar u^{(n)} (t) \|_{L_x^2} \\
& \quad + \| D^\beta h_1(t) \|_{L_x^\infty} \cdot \| D^{\alpha+2-\beta} \bar u^{(n)} (t) \|_{L_x^2} \cdot \| \partial_t
D^\alpha \bar u^{(n)} (t) \|_{L_x^2} \\
& \quad +\sum_{|\alpha| \le m-1} \sum_{\substack{ |\beta| \le |\alpha|+1 \\ |\beta| \le \frac m2} }
\| D^\beta \bar \vv^{(n-1)}(t) \|_{L_x^\infty} \cdot \|D^{\alpha+2-\beta} \bar \vv^{(n)} (t) \|_{L_x^2}
\cdot \| \partial_t D^\alpha \bar u^{(n)} (t) \|_{L_x^2} \\
& \quad + \| D^\beta \bar u^{(n-1)} (t) \|_{L_x^\infty} \cdot \| D^{\alpha+2-\beta} \bar u^{(n)} (t) \|_{L_x^2} \cdot \| \partial_t
D^\alpha \bar u^{(n)} (t) \|_{L_x^2} \\
& \quad +\sum_{|\alpha| \le m-1} \sum_{\substack{ |\beta| \le |\alpha|+1 \\ |\beta| > \frac m2} }
\| D^\beta \bar \vv^{(n-1)}(t) \|_{L_x^2} \cdot \|D^{\alpha+2-\beta} \bar \vv^{(n)} (t) \|_{L_x^\infty}
\cdot \| \partial_t D^\alpha \bar u^{(n)} (t) \|_{L_x^2} \\
& \quad + \| D^\beta \bar u^{(n-1)} (t) \|_{L_x^2} \cdot \| D^{\alpha+2-\beta} \bar u^{(n)} (t) \|_{L_x^\infty} \cdot \| \partial_t
D^\alpha \bar u^{(n)} (t) \|_{L_x^2} \\
& \lsm \frac {\epsilon} {\langle t \rangle} A_n(t)^2.
\end{align*}

The term \eqref{52ab} will be absorbed into the LHS of \eqref{52a}. To bound \eqref{52ac}, note that
\begin{align*}
| \chi_n^\prime (t) | \le 2^{-n} | \chi^\prime( 2^{-n} t) | \lsm 1, \quad \forall\, n\ge 0, \; t\ge 0.
\end{align*}

By using integration by parts, \eqref{e100}--\eqref{e100c} and Sobolev embedding, we obtain
\begin{align*}
| \eqref{52ac} | &\lsm  \sum_{\substack{|\alpha| =m \\ 1\le i,j\le 1}}
\Bigl( \| \vv^{(n-1)}_j (t) \|_{L_x^\infty} + \| \partial_{x_j} \vv^{(n-1)}_j (t) \|_{L_x^\infty} \Bigr) A_n(t)^2 \\
& \lsm (\| h_2(t)\|_{L_x^\infty} + \| D h_2(t) \|_{L_x^\infty} + \| \Phi_2(t)\|_{H^m} + \| \bar \vv^{(n-1)}(t)\|_{H^m} )
A_n(t)^2 \\
& \lsm \frac {\epsilon} {\langle t \rangle} A_n(t)^2.
\end{align*}

Proceeding in a similar manner, we also obtain
\begin{align*}
| \eqref{52ad} | \lsm \frac {\epsilon} {\langle t \rangle} A_n(t)^2.
\end{align*}

For \eqref{52ae} we shall integrate by parts twice and this gives us
\begin{align}
\eqref{52ae} &= -\chi_n(t) \sum_{\substack{|\alpha|=m \\ 1\le i,j\le 2}} \int_{\R^2} \vv_j^{(n-1)} \partial_t D^\alpha
\bar \vv_i^{(n)} \partial_{x_j} \partial_{x_i} D^\alpha \bar u^{(n)} dx \label{52ae1} \\
& \quad + \chi_n(t) \sum_{\substack{|\alpha|=m \\ 1\le i,j\le 2}} \int_{\R^2}
\partial_{x_j} \vv_j^{(n-1)} \partial_t D^\alpha \partial_{x_i}
\bar \vv_i^{(n)} D^\alpha \bar u^{(n)} dx  \label{52ae2} \\
& \quad -\chi_n(t) \sum_{\substack{|\alpha|=m \\ 1\le i,j\le 2}} \int_{\R^2}
\partial_{x_i} \vv_j^{(n-1)} \partial_t D^\alpha
\bar \vv_i^{(n)} \partial_{x_j} D^\alpha \bar u^{(n)} dx \label{52ae3}
\end{align}

We shall keep the term \eqref{52ae1} since it will cancel with the corresponding term obtained in the calculation
of $\bar \vv^{(n)}$ later. The term \eqref{52ae2} can be bounded after one more integration by parts in $\partial_{x_i}$ and
this will give us
\begin{align*}
|\eqref{52ae2} | \lsm \frac {\epsilon} {\langle t \rangle} A_n(t)^2.
\end{align*}

The term \eqref{52ae3} can be bounded in a similar way as \eqref{52ac} and we get
\begin{align*}
|\eqref{52ae3} | \lsm \frac {\epsilon} {\langle t \rangle} A_n(t)^2.
\end{align*}

Collecting all the estimates, we obtain
\begin{align*}
\eqref{52a} & =\frac d {dt} \Bigl( \chi_n(t) \sum_{\substack{|\alpha|=m \\ 1\le i,j\le 2 }}
\int_{\mathbb R^2} \vv_j^{(n-1)} D^\alpha \partial_{x_j} \partial_{x_i} \bar \vv_i^{(n)} D^\alpha \bar u^{(n)} dx \Bigr) \\
&\quad -\chi_n(t) \sum_{\substack{|\alpha|=m \\ 1\le i,j\le 2}} \int_{\R^2} \vv_j^{(n-1)} \partial_t D^\alpha
\bar \vv_i^{(n)} \partial_{x_j} \partial_{x_i} D^\alpha \bar u^{(n)} dx \\
& \quad + E_1,
\end{align*}
where
\begin{align*}
|E_1| \lsm \frac {\epsilon} {\langle t \rangle} A_n(t)^2.
\end{align*}

Next we consider \eqref{52b} and write it as
\begin{align}
\eqref{52b} & = - \chi_n(t) \sum_{|\alpha| \le m-1} \int_{\R^2} D^\alpha \partial_t
\Bigl( (\vv^{(n-1)} \cdot \nabla ) \bar u^{(n)} + u^{(n-1)} \nabla \cdot ( \bar \vv^{(n)} ) \Bigr) \partial_t D^\alpha \bar u^{(n)} dx
\label{52ba} \\
& \quad - \chi_n(t) \sum_{|\alpha|=m} \int_{\R^2} D^\alpha \partial_t \Bigl( (\vv^{(n-1)} \cdot \nabla) \bar u^{(n)}  \Bigr)
\partial_t D^\alpha \bar u^{(n)} dx \label{52bb} \\
& \quad - \chi_n(t) \sum_{|\alpha|=m} \int_{\R^2} D^\alpha \partial_t \Bigl( u^{(n-1)} \nabla \cdot (\bar \vv^{(n)} ) \Bigr)
\partial_t D^\alpha \bar u^{(n)} dx. \label{52bc}
\end{align}

The term \eqref{52ba} can be easily bounded by
\begin{align*}
| \eqref{52ba} | \lsm \frac {\epsilon} {\langle t \rangle} A_n(t)^2.
\end{align*}

For \eqref{52bb}, by using integration by parts, \eqref{e100}--\eqref{e100c}, Sobolev embedding and proceeding in the similar way as
in the estimate of \eqref{52a}, we obtain
\begin{align*}
| \eqref{52bb} | & \lsm \sum_{|\alpha|=m} \Bigl| \int_{\R^2} D^\alpha \Bigl( (\partial_t \vv^{(n-1)} \cdot \nabla) \bar u^{(n)} \Bigr) \partial_t
D^\alpha \bar u^{(n)} dx \Bigr| \\
& \quad + \sum_{|\alpha| =m } \Bigl|
\int_{\R^2} D^\alpha \Bigl( \vv^{(n-1)} \cdot \nabla \partial_t \bar u^{(n)} \Bigr) \partial_t D^\alpha \bar u^{(n)} dx \Bigr| \\
& \lsm \sum_{|\alpha|=m} \sum_{|\beta| \le |\alpha|} \int_{\R^2} |\partial_t D^{\alpha-\beta} \vv^{(n-1)} | \cdot |\nabla D^\beta \bar u^{(n)} |
\cdot | \partial_t D^\alpha \bar u^{(n)} | dx \\
& \quad + \sum_{|\alpha|=m} \sum_{|\beta| < |\alpha|} \int_{\R^2} |D^{\alpha-\beta} \vv^{(n-1)} |
\cdot | \nabla D^\beta \partial \bar u^{(n)} | \cdot | \partial_t D^\alpha \bar u^{(n)} | dx \\
&\quad + \sum_{|\alpha|=m} \Bigl| \int_{\R^2} (\vv^{(n-1)} \cdot \nabla )D^\alpha \partial_t \bar u^{(n)} \cdot \partial_t D^\alpha
\bar u^{(n)} dx \Bigr| \\
& \lsm \frac {\epsilon}{\langle t \rangle} A_n(t)^2.
\end{align*}

For \eqref{52bc}, we use Leibniz in time and write it as
\begin{align}
\eqref{52bc} & = -\sum_{|\alpha| =m } \int_{\R^2} D^\alpha \Bigl( \partial_t u^{(n-1)} \nabla \cdot \bar \vv^{(n)} \Bigr) \partial_t D^\alpha
\bar u^{(n)} dx \label{52bca} \\
& \quad - \sum_{|\alpha|=m} \int_{\R^2} D^\alpha \Bigl( u^{(n-1)} \partial_t \nabla \cdot \bar \vv^{(n)} \Bigr) \partial_t D^\alpha
\bar u^{(n)} dx. \label{52bcb}
\end{align}

We can easily bound \eqref{52bca} as
\begin{align*}
|\eqref{52bca} | & \lsm \sum_{|\alpha| =m } \sum_{|\beta| \le |\alpha|} \int_{\R^2} |D^\beta \partial_t u^{(n-1)} |
\cdot | D^{\alpha-\beta} \nabla \cdot \bar \vv^{(n)} | \cdot | \partial_t D^\alpha \bar u^{(n)} | dx \\
& \lsm \frac {\epsilon}{\langle t \rangle } A_n(t)^2.
\end{align*}

On the other hand for \eqref{52bcb}, we further decompose it into
\begin{align}
\eqref{52bcb} & = -\sum_{|\alpha|=m} \int_{\R^2} u^{(n-1)} \partial_t D^\alpha \nabla \cdot \bar \vv^{(n)} \partial_t D^\alpha \bar u^{(n)} dx \notag \\
& \qquad - \sum_{|\alpha| =m } \sum_{0<|\beta|\le |\alpha|} \binom {\alpha}{\beta} \int_{\R^2}
D^\beta u^{(n-1)} D^{\alpha-\beta} \partial_t \nabla \cdot \bar \vv^{(n)} \partial_t D^\alpha \bar u^{(n)} dx \notag\\
& = \sum_{|\alpha|=m} \int_{\R^2} u^{(n-1)} \partial_t D^\alpha \bar \vv^{(n)} \cdot \nabla \partial_t D^\alpha \bar u^{(n)} dx \label{52bcba} \\
& \qquad +\sum_{|\alpha|=m} \int_{\R^2} \Bigl( \nabla u^{(n-1)} \cdot \partial_t D^\alpha \bar \vv^{(n)} \Bigr)
\partial_t D^\alpha \bar u^{(n)} dx \label{52bcbb} \\
& \qquad - \sum_{|\alpha|=m} \sum_{0<|\beta|\le |\alpha|}
\binom{\alpha}{\beta} \int_{\R^2} D^\beta u^{(n-1)} D^{\alpha-\beta} \partial_t \nabla \cdot \bar \vv^{(n)} \partial_t D^{\alpha}
\bar u^{(n)} dx. \label{52bcbc}
\end{align}

We shall keep the term \eqref{52bcba} since it will cancel with a corresponding term arising from the calculation of
$\bar \vv^{(n)}$ later. The terms \eqref{52bcbb}, \eqref{52bcbc} can be bounded in the same way as the calculation of
\eqref{52a} and this gives us
\begin{align*}
|\eqref{52bcbb}| + |\eqref{52bcbc}| \lsm \frac {\epsilon}{\langle t \rangle} A_n(t)^2.
\end{align*}

Collecting all the estimates, we obtain
\begin{align*}
\eqref{52b} &= \sum_{|\alpha|=m} \int_{\R^2} u^{(n-1)} \partial_t D^\alpha \bar \vv^{(n)} \cdot \nabla \partial_t D^\alpha \bar u^{(n)} dx \\
&\qquad + E_2,
\end{align*}
where
\begin{align*}
|E_2| \lsm \frac {\epsilon} {\langle t \rangle} A_n(t)^2.
\end{align*}

We still have to bound the terms \eqref{52c}--\eqref{52f}. The estimates of these terms can be done in a similar manner as that of
\eqref{52a}. We omit the tedious computations and write the final result as
\begin{align*}
|\eqref{52c}| + |\eqref{52d}| + |\eqref{52e}| + |\eqref{52f}| \lsm \frac {\epsilon^2} {\langle t \rangle^2} A_n(t).
\end{align*}

Putting together all the estimates, we have the expression
\begin{align}
 & \frac 12 \frac d {dt} \Bigl( \sum_{|\alpha| \le m} \| D^\alpha \bar u^{(n)} (t) \|^2_{L_x^2} + \| \left( D^\alpha \bar u^{(n)} (t) \right)^\prime
 \|_{L_x^2}^2 \Bigr) \notag \\
  = & \; \frac d {dt} \Bigl( \chi_n(t) \sum_{\substack{|\alpha|=m \\ 1\le i,j\le 2 }}
\int_{\mathbb R^2} \vv_j^{(n-1)} D^\alpha \partial_{x_j} \partial_{x_i} \bar \vv_i^{(n)} D^\alpha \bar u^{(n)} dx \Bigr) \notag \\
&\quad -\chi_n(t) \sum_{\substack{|\alpha|=m \\ 1\le i,j\le 2}} \int_{\R^2} \vv_j^{(n-1)} \partial_t D^\alpha
\bar \vv_i^{(n)} \partial_{x_j} \partial_{x_i} D^\alpha \bar u^{(n)} dx \notag \\
& \quad +\sum_{|\alpha|=m} \int_{\R^2} u^{(n-1)} \partial_t D^\alpha \bar \vv^{(n)} \cdot \nabla \partial_t D^\alpha \bar u^{(n)} dx \notag \\
& \quad + F_1+F_2, \label{e61a}
\end{align}
where
\begin{align*}
|F_1(t)| &\lsm \frac {\epsilon}{\langle t \rangle} A_n(t)^2, \\
|F_2(t)| & \lsm \frac {\epsilon^2} {\langle t \rangle^2} A_n(t).
\end{align*}

Next we consider the bounds on $\bar \vv^{(n)}$ by using \eqref{eq10b}. Proceeding in a similar fashion as in the computation of
$\bar u^{(n)}$, we obtain
\begin{align}
& \frac 12 \frac d {dt} \Bigl( \sum_{|\alpha| \le m} \| D^\alpha \bar \vv^{(n)} (t) \|^2_{L_x^2} + \| \left( D^\alpha \bar \vv^{(n)} (t) \right)^\prime
 \|_{L_x^2}^2 \Bigr) \notag \\
 = & \chi_n(t) \sum_{\substack{|\alpha|=m \\ 1\le i,j\le 2}} \int_{\R^2} \vv_j^{(n-1)} \partial_{x_i}\partial_{x_j}D^\alpha \bar u^{(n)} \cdot D^\alpha
 \partial_t \bar \vv_i^{(n)} dx \notag \\
 & \quad + \frac d {dt} \Bigl( \chi_n(t) \sum_{|\alpha|=m} \int_{\R^2} u^{(n-1)} | D^\alpha \nabla \cdot \bar \vv^{(n)} |^2 dx \Bigr) \notag \\
 & \quad -\sum_{|\alpha|=m} \chi_n(t) \int_{\R^2} u^{(n-1)} \nabla \partial_t D^\alpha \bar u^{(n)} \cdot D^\alpha \partial_t \bar \vv^{(n)} dx \notag \\
 & \quad +F_3+F_4, \label{e61b}
 \end{align}
 where
\begin{align*}
|F_3(t)| &\lsm \frac {\epsilon}{\langle t \rangle} A_n(t)^2, \\
|F_4(t)| & \lsm \frac {\epsilon^2} {\langle t \rangle^2} A_n(t).
\end{align*}

Denote
\begin{align}
B_n(t)^2 &= \frac 12 \Bigl( \sum_{|\alpha|\le m} \| D^\alpha \bar u^{(n)} (t) \|_{L_x^2}^2 + \| \left( D^\alpha \bar u^{(n)} (t) \right)^\prime \|_{L_x^2}^2
\notag \\
&\qquad + \sum_{|\alpha|\le m} \| D^\alpha \bar \vv^{(n)} (t) \|^2_{L_x^2} + \| \left( D^\alpha \bar \vv^{(n)} (t) \right)^\prime \|^2_{L_x^2}
\Bigr) \notag \\
& \quad + \chi_n(t) \sum_{\substack{|\alpha| =m \\ 1\le i,j\le 2}} \int_{\R^2} \partial_{x_j} \vv^{(n-1)}_{j} D^\alpha \partial_{x_i}
\bar \vv_i^{(n)} D^\alpha \bar u^{(n)} dx \label{e94a} \\
& \quad + \chi_n(t) \sum_{ \substack{ |\alpha|=m \\ 1\le i,j\le 2} } \int_{\R^2} \vv_j^{(n-1)} D^\alpha \partial_{x_i} \bar \vv^{(n)}_i
\partial_{x_j} D^\alpha \bar u^{(n)} dx \label{e94b} \\
& \quad - \chi_n(t) \sum_{|\alpha|=m} \int_{\R^2} u^{(n-1)} | D^\alpha \nabla \cdot \bar \vv^{(n)} |^2 dx. \label{e94c}
\end{align}

By using \eqref{e100}--\eqref{e100c}, we have
\begin{align*}
& | \eqref{e94a} | + |\eqref{e94b}| + | \eqref{e94c} | \\
\lsm & \epsilon \cdot \frac 12 \Bigl( \sum_{|\alpha|\le m} \| D^\alpha \bar u^{(n)} (t) \|_{L_x^2}^2 + \| \left( D^\alpha \bar u^{(n)} (t) \right)^\prime \|_{L_x^2}^2
\notag \\
&\qquad + \sum_{|\alpha|\le m} \| D^\alpha \bar \vv^{(n)} (t) \|^2_{L_x^2} + \| \left( D^\alpha \bar \vv^{(n)} (t) \right)^\prime \|^2_{L_x^2}
\Bigr).
\end{align*}

Therefore for $\epsilon$ sufficiently small (depending only on an absolute constant), we have
\begin{align}
\frac 1 C B_n(t) \le A_n(t) \le C B_n(t), \label{e107}
\end{align}
where $C$ is some absolute constant. Adding together \eqref{e61a} and \eqref{e61b} and using \eqref{e107}, we have
\begin{align}
\Bigl| \frac d {dt} B_n(t) \Bigr| \lsm \frac {\epsilon} {\langle t \rangle } B_n(t) + \frac {\epsilon^2} {\langle t \rangle^2},
\quad t\ge 0, \label{e102a}
\end{align}
where we have canceled a factor of $B_n(t)$ on both sides of the inequality.

By \eqref{eApriori} and \eqref{e107}, we have the rough estimate
\begin{align}
|B_n(t)| \lsm \frac {K_n} {\langle t \rangle}, \quad\forall\, t\ge 0, \label{e102b}
\end{align}
where $K_n$ is a constant depending on the iteration number $n$.

We now need the following

\begin{lem}[Infinite time Gronwall lemma] \label{gron_1}
Let B(t) be a nonnegative smooth function on $[1,\infty)$ satisfying
\begin{align} \label{e905}
\left|\frac d {dt} B(t) \right| \le \frac \epsilon{t} B(t) + \frac {\delta} {t^2}
\end{align}
and
\begin{align}
B(t) \le \frac {K}{t}, \quad \forall\, t \ge 1. \label{e904}
\end{align}
Here $0<\epsilon<1$, $\delta>0$, $K>0$ are constants. Then
\begin{align}
B(t) \le \frac {\delta} {1-\epsilon} \cdot  \frac 1 t, \quad\, \forall\, t\ge 1. \label{e906}
\end{align}
\end{lem}
\begin{rem}
The crucial point here is that the final bound on $B(t)$ is independent of the constant $K$.
\end{rem}
\begin{proof}
By Fundamental Theorem of Calculus and plugging the estimate \eqref{e904} into \eqref{e905}, we obtain
\begin{align*}
B(t) &\le \int_{t}^\infty \left| \frac d{d\tau} B(\tau) \right| d\tau \\
 & \le \int_t^\infty \Bigl( \frac \epsilon {\tau} \cdot \frac {K}{\tau} + \frac {\delta} {\tau^2} \Bigr) d\tau \\
 & \le \frac { K\epsilon+ \delta}{t}.
 \end{align*}
We inductively assume $j\ge0$,
\begin{align*}
B(t) & \le \frac { K\epsilon^{j+1} + \delta (1+\cdots+\epsilon^j) }{t}.
\end{align*}
Then for $j+1$, using \eqref{e905} we have
\begin{align*}
\left| \frac d {dt} B(t) \right| &\le \frac 1 {t^2} \cdot \Bigl( \epsilon ( K \epsilon^{j+1} + \delta(1+ \cdots +\epsilon^j)) + \delta \Bigr) \\
& \le \frac 1 {t^2} \Bigl( K \epsilon^{j+2} + \delta(1+ \cdots + \epsilon^{j+1}) \Bigr).
\end{align*}
Using again the fundamental theorem of calculus we conclude that
the inductive estimate holds for all $j$. Taking $j\to \infty$ we obtain \eqref{e906}.
\end{proof}

By Lemma \ref{gron_1} and \eqref{e102a}, \eqref{e102b}, we obtain for $\epsilon$ sufficiently small
\begin{align}
B_n(t) \lsm \frac {\epsilon^2} {\langle t \rangle}. \label{e115}
\end{align}
The extra power of $\epsilon$ can be used to kill the underlying implied constants. Therefore the inductive
assumption \eqref{e100} also holds for step $n$ with the same underlying constant. Hence we obtain the following:
there exists a constant $C>0$, such that for all $n\ge 0$,
\begin{align}
\| \bar u^{(n)} (t) \|_{H^m} + \| (\bar u^{(n)} (t) )^\prime \|_{H^m}
+ \| \bar \vv^{(n)}(t) \|_{H^m} \notag \\
 + \| ( \bar \vv^{(n)} (t) )^\prime \|_{H^m}
\le C\frac {\epsilon}{\langle t \rangle},\quad\forall\, t\ge 0. \label{e127}
\end{align}

\begin{rem}
Since the bound \eqref{e115} is slightly stronger than the inductive bound \eqref{e100}, one may wonder whether
the bound \eqref{e100} is optimal in terms of the power of $\epsilon$. Indeed by revisiting the inequality \eqref{e102a}
and Lemma \ref{gron_1}, one can easily derive the bound
\begin{align*}
B_n(t) \lsm \frac {\epsilon^2} {\langle t \rangle}, \quad\forall\, n\ge 0,
\end{align*}
from which it follows that
\begin{align*}
\| \bar u^{(n)} (t) \|_{H^m} + \| (\bar u^{(n)} (t) )^\prime \|_{H^m}
+ \| \bar \vv^{(n)}(t) \|_{H^m} \notag \\
 + \| ( \bar \vv^{(n)} (t) )^\prime \|_{H^m}
\le C\frac {\epsilon^2}{\langle t \rangle},\quad\forall\, t\ge 0, \; n\ge 0.
\end{align*}
\end{rem}
\section{Contraction estimates and existence of the smooth solution} \label{sec_contraction}
In this section we shall show that the sequence $(\bar u^{(n)}, \bar \vv^{(n)})$ forms
a Cauchy sequence in some suitable metric. Define
\begin{align*}
U^{(n)} & = \bar u^{(n)} - \bar u^{(n-1)} \\
V^{(n)} & = \bar \vv^{(n)} -\bar \vv^{(n-1)}.
\end{align*}
From \eqref{eq10a}, \eqref{eq10b}, we get
\begin{align}
(\square +1 ) U^{(n)} & = \chi_n(t) \nabla \cdot \Bigl( (V^{(n-1)} \cdot \nabla) \bar \vv^{(n)} + U^{(n-1)} \nabla \bar u^{(n)} \Bigr) \notag \\
& \quad + \chi_n(t) \nabla \cdot \Bigl( \vv^{(n-2)} \cdot \nabla V^{(n)} + u^{(n-2)} \nabla U^{(n)} \Bigr) \notag \\
& \quad + (\chi_n(t) - \chi_{n-1}(t) ) \nabla \cdot \Bigl( (\vv^{(n-2)} \cdot \nabla) \bar \vv^{(n-1)} + u^{(n-2)} \nabla \bar u^{(n-1)} \Bigr)
\notag \\
& \quad - \chi_n(t) \cdot \partial_t \Bigl( (V^{(n-1)} \cdot \nabla) \bar u^{(n)} + U^{(n-1)} \nabla \cdot \bar \vv^{(n)} \Bigr) \notag \\
& \quad - \chi_n(t) \cdot \partial_t \Bigl( (\vv^{(n-2)} \cdot \nabla) U^{(n)} + u^{(n-2)} \nabla \cdot V^{(n)} \Bigr) \notag \\
& \quad - (\chi_n(t) -\chi_{n-1}(t)) \partial_t \Bigl( (\vv^{(n-2)} \cdot \nabla) \bar u^{(n-1)} + u^{(n-2)} \nabla \cdot \bar \vv^{(n-1)}
\Bigr) \notag \\
& \quad + \nabla \cdot \Bigl( (V^{(n-1)} \cdot \nabla) \Phi_2 + U^{(n-1)} \nabla \Phi_1 \Bigr) \notag \\
& \quad - \partial_t \Bigl( (V^{(n-1)} \cdot \nabla) \Phi_1 + U^{(n-1)} \nabla \cdot \Phi_2 \Bigr) \notag \\
& \quad + \nabla \cdot \Bigl(
\bigl( V^{(n-1)}  \cdot \nabla \bigr) h_2 + U^{(n-1)} \nabla h_1 \Bigr) \notag \\
& \quad - \partial_t
\Bigl( \bigl( V^{(n-1)}  \cdot \nabla \bigr) h_1 + U^{(n-1)}  \nabla \cdot h_2 \Bigr), \label{Nov12a}
\end{align}

and

\begin{align}
(\square +1) V^{(n)} & = \chi_n(t) \nabla \Bigl( (V^{(n-1)} \cdot \nabla)(\bar u^{(n)}) + U^{(n-1)} \nabla \cdot \bar \vv^{(n)} \Bigr) \notag \\
&\quad+ \chi_n(t) \nabla \Bigl( (\vv^{(n-2)} \cdot \nabla)(U^{(n)}) + u^{(n-2)} \nabla \cdot V^{(n)} \Bigr) \notag \\
&\quad +(\chi_n(t)-\chi_{n-1}(t)) \nabla \Bigl( (\vv^{(n-2)} \cdot \nabla)(\bar u^{(n-1)}) + u^{(n-2)} \nabla \cdot \bar \vv^{(n-1)} \Bigr) \notag \\
& \quad - \chi_n(t) \partial_t
\Bigl( (V^{(n-1)} \cdot \nabla) ( \bar \vv^{(n)}) + U^{(n-1)} \nabla ( \bar u^{(n)} ) \Bigr) \notag \\
& \quad - \chi_n(t) \partial_t
\Bigl( (\vv^{(n-2)} \cdot \nabla) ( V^{(n)}) + u^{(n-2)} \nabla ( U^{(n)} ) \Bigr) \notag \\
& \quad - (\chi_n(t)-\chi_{n-1}(t))  \partial_t
\Bigl( (\vv^{(n-2)} \cdot \nabla) ( \bar \vv^{(n-1)}) + u^{(n-2)} \nabla ( \bar u^{(n-1)} ) \Bigr) \notag \\
& \quad + \nabla \Bigl( (V^{(n-1)} \cdot \nabla) \Phi_1 + U^{(n-1)} \nabla \cdot \Phi_2 \Bigr) \notag \\
& \quad -\partial_t \Bigl( (V^{(n-1)} \cdot \nabla ) \Phi_2 + U^{(n-1)} \nabla \Phi_1 \Bigr) \notag \\
& \quad + \nabla \Bigl(  V^{(n-1)}  \cdot \nabla h_1 + U^{(n-1)} \nabla \cdot h_2 \Bigr) \notag \\
& \quad - \partial_t \Bigl(
\bigl( V^{(n-1)}  \cdot \nabla \bigr) h_2 + U^{(n-1)}  \nabla h_1 \Bigr) \notag \\
& \qquad - \nabla \Delta^{-1} \nabla \cdot( U^{(n-1)} \vv^{(n-1)} ) \notag \\
& \qquad - \nabla \Delta^{-1} \nabla \cdot( u^{(n-2)} V^{(n-1)} ) \label{Nov12b}
\end{align}

For the convenience of notations, we define
\begin{align}
w_n(t)^2 = \| U^{(n)} (t) \|_{L_x^2}^2 + \| ( U^{(n)}(t) )^\prime \|_{L_x^2}^2+ \| V^{(n)}(t)\|_{L_x^2}^2 \notag \\
+ \| (V^{(n)}(t) )^\prime \|_{L_x^2}^2, \quad n\ge 0. \label{132}
\end{align}

We shall derive some recursive inequalities for $w_n(t)$ from which it will follow that the sequence
$(\bar u^{(n)}, \bar \vv^{(n)})$ is a contraction.

Multiplying both sides of \eqref{Nov12a} by $\partial_t U^{(n)}$ and integrating by parts, we obtain
\begin{align}
 & \frac d {dt} \Bigl( \frac 12 \| U^{(n)} (t) \|_{L_x^2}^2 + \frac 12 \| ( U^{(n)} (t) )^\prime \|_{L_x^2}^2 \Bigr) \notag \\
= & \int_{\R^2} \chi_n(t) \nabla \cdot \Bigl( (V^{(n-1)} \cdot \nabla) \bar \vv^{(n)} + U^{(n-1)} \nabla \bar u^{(n)} \Bigr)
\cdot \partial_t U^{(n)} dx \label{113a} \\
& \quad +
\chi_n(t) \int_{\R^2}
\nabla \cdot \Bigl( (\vv^{(n-2)} \cdot \nabla) V^{(n)} + u^{(n-2)} \nabla U^{(n)} \Bigr)
\cdot \partial_t U^{(n)} dx \label{113b} \\
& \quad + (\chi_n(t) - \chi_{n-1}(t) )
\int_{\R^2} \nabla \cdot \Bigl( (\vv^{(n-2)} \cdot \nabla) \bar \vv^{(n-1)} + u^{(n-2)} \nabla \bar u^{(n-1)} \Bigr)
\partial_t U^{(n)} dx
\label{113c} \\
& \quad - \chi_n(t) \int_{\R^2} \partial_t \Bigl( (V^{(n-1)} \cdot \nabla) \bar u^{(n)} + U^{(n-1)} \nabla \cdot \bar \vv^{(n)} \Bigr)
\partial_t U^{(n)} dx
\label{113d} \\
& \quad - \chi_n(t) \int_{\R^2}
\partial_t \Bigl( (\vv^{(n-2)} \cdot \nabla) U^{(n)} + u^{(n-2)} \nabla \cdot V^{(n)} \Bigr)
\partial_t U^{(n)} dx \label{113e} \\
& \quad - (\chi_n(t) -\chi_{n-1}(t))
\int_{\R^2} \partial_t \Bigl( (\vv^{(n-2)} \cdot \nabla) \bar u^{(n-1)} + u^{(n-2)} \nabla \cdot \bar \vv^{(n-1)}
\Bigr) \partial_t U^{(n)} dx \label{113f} \\
& \quad + \int_{\R^2} \nabla \cdot \Bigl( (V^{(n-1)} \cdot \nabla) \Phi_2 + U^{(n-1)} \nabla \Phi_1 \Bigr) \partial_t U^{(n)} dx \label{113g} \\
& \quad - \int_{\R^2} \partial_t \Bigl( (V^{(n-1)} \cdot \nabla) \Phi_1 + U^{(n-1)} \nabla \cdot \Phi_2 \Bigr)
\partial_t U^{(n)} dx \label{113h} \\
& \quad + \int_{\R^2} \nabla \cdot \Bigl(
\bigl( V^{(n-1)}  \cdot \nabla \bigr) h_2 + U^{(n-1)} \nabla h_1 \Bigr) \partial_t U^{(n)} dx \label{113i} \\
& \quad - \int_{\R^2} \partial_t
\Bigl( \bigl( V^{(n-1)}  \cdot \nabla \bigr) h_1 + U^{(n-1)}  \nabla \cdot h_2 \Bigr) \partial_t U^{(n)} dx, \label{113j}
\end{align}

By using \eqref{e127}, \eqref{e100a}--\eqref{e100c} and \eqref{132}, we have
\begin{align*}
& | \eqref{113a} | + | \eqref{113d}|+ | \eqref{113g}|+ |\eqref{113h}| +|\eqref{113i}|+|\eqref{113j}| \\
\lsm & \frac {\epsilon}{\langle t \rangle} w_{n-1}(t) \cdot w_n(t).
\end{align*}

Also we have
\begin{align*}
& | \eqref{113c} | + | \eqref{113f} | \\
\lsm & | \chi_n(t) -\chi_{n-1}(t) | \cdot \frac {\epsilon^2} {\langle t \rangle^2} \cdot w_n(t).
\end{align*}

By using integration by parts, we write \eqref{113b} as
\begin{align}
\eqref{113b} & = -\chi_n(t) \int_{\R^2} u^{(n-2)} \nabla U^{(n)} \cdot \partial_t \nabla U^{(n)} dx \label{113ba} \\
& \quad - \chi_n (t) \sum_{1\le i,j \le 2} \int_{\R^2} \vv_i^{(n-2)} \partial_{x_i} V_j^{(n)} \partial_t \partial_{x_j} U^{(n)} dx.
\label{113bb}
\end{align}

For \eqref{113ba} we use Leibniz in time and this gives
\begin{align}
\eqref{113ba} &= - \frac d {dt} \Bigl( \frac 12 \chi_n(t) \int_{\R^2} u^{(n-2)} | \nabla U^{(n)} |^2 dx \Bigr) \label{113baa} \\
&\quad + \frac 12 \chi_n^\prime(t) \int_{\R^2} u^{(n-2)} | \nabla U^{(n)} |^2 dx \label{113bab} \\
& \quad + \frac 12 \chi_n(t) \int_{\R^2} \partial_t u^{(n-2)} |\nabla U^{(n)} |^2 dx. \label{113bac}
\end{align}

The first term \eqref{113baa} will be absorbed into the LHS of \eqref{113a}. By using \eqref{e127}, \eqref{e100a}--\eqref{e100c},
the other two terms \eqref{113bab}, \eqref{113bac}
can be easily bounded as
\begin{align*}
|\eqref{113bab}| + | \eqref{113bac} | \lsm \frac {\epsilon} {\langle t \rangle} w_n(t)^2.
\end{align*}

Next we consider the term \eqref{113bb}. By using repeatedly integration by parts, we have
\begin{align}
\eqref{113bb} & = \chi_n(t) \sum_{1\le i,j\le 2} \int_{\R^2} \partial_{x_i} \vv_i^{(n-2)} V_j^{(n)} \partial_t \partial_{x_j} U^{(n)} dx \notag \\
& \qquad + \chi_n(t) \sum_{1\le i,j\le 2} \int_{\R^2} \vv_i^{(n-2)} V_j^{(n)} \partial_t \partial_{x_j} \partial_{x_i} U^{(n)} dx \notag \\
& = -\chi_n(t) \sum_{1\le i,j\le 2} \int_{\R^2} \partial_{x_i} \vv_i^{(n-2)} \partial_{x_j} V_j^{(n)} \partial_t U^{(n)} dx \label{118a} \\
& \quad -\chi_n(t) \sum_{1\le i,j\le 2} \int_{\R^2} \partial_{x_j} \partial_{x_i} \vv_i^{(n-2)} V_j^{(n)} \partial_t U^{(n)} dx \label{118b} \\
& \quad - \frac d {dt} \Bigl( \chi_n(t)  \sum_{1\le i,j\le 2} \int_{\R^2} \partial_{x_j} \vv_i^{(n-2)} V_j^{(n)} \partial_{x_i} U^{(n)} dx \notag \\
&\qquad + \chi_n(t)\sum_{1\le i,j\le 2} \int_{\R^2} \vv_i^{(n-2)} \partial_{x_j} V_j^{(n)} \partial_{x_i} U^{(n)} dx \Bigr) \label{118c} \\
& \quad + \chi_n^\prime (t)  \sum_{1\le i,j\le 2} \int_{\R^2} \partial_{x_j} \vv_i^{(n-2)} V_j^{(n)} \partial_{x_i} U^{(n)} dx \label{118d} \\
&\quad + \chi_n^\prime(t)\sum_{1\le i,j\le 2} \int_{\R^2} \vv_i^{(n-2)} \partial_{x_j} V_j^{(n)} \partial_{x_i} U^{(n)} dx \label{118e} \\
&\quad + \chi_n(t) \sum_{1\le i,j\le 2} \int_{\R^2} \partial_{x_j} \partial_t \vv_i^{(n-2)} V_j^{(n)} \partial_{x_i} U^{(n)} dx \label{118f} \\
& \quad + \chi_n(t) \sum_{1\le i,j\le 2} \int_{\R^2} \partial_t \vv_i^{(n-2)} \partial_{x_j} V_j^{(n)} \partial_{x_i} U^{(n)} dx \label{118g} \\
& \quad - \chi_n(t) \sum_{1\le i,j\le 2} \int_{\R^2} \partial_t V_j^{(n)} \partial_{x_j} \Bigl( \partial_{x_i} U^{(n)} \cdot \vv_i^{(n-2)} \Bigr) dx
\label{118h} \\
& \quad + \chi_n(t) \sum_{1\le i,j\le 2} \int_{\R^2} \partial_t V_j^{(n)} \partial_{x_i} U^{(n)} \partial_{x_j} \vv_i^{(n-2)} dx \label{118i}
\end{align}

We shall absorb \eqref{118c} into the LHS of \eqref{e127}. We will make no changes to the term \eqref{118h} since it will cancel with a corresponding
term in the calculation of $V^{(n)}$. For the rest of the terms, we use again \eqref{e127}, \eqref{e100a}--\eqref{e100c} to bound them as
\begin{align*}
& |\eqref{118a}| + | \eqref{118b}|+|\eqref{118d}|+|\eqref{118e}|+|\eqref{118f}|+|\eqref{118g}|+|\eqref{118i}| \\
\lsm & \frac {\epsilon}{\langle t \rangle} w_n(t)^2.
\end{align*}

This ends the estimate of \eqref{113bb} and \eqref{113b}. By a similar but tedious calculation, we also have
\begin{align*}
\eqref{113e} &= \chi_n(t) \int_{\R^2} \partial_t \Bigl( u^{(n-2)} \nabla U^{(n)} \Bigr) \cdot \partial_t V^{(n)} dx \\
& \qquad + \tilde E_1,
\end{align*}
where
\begin{align*}
|\tilde E_1| \lsm \frac {\epsilon} {\langle t \rangle} w_n(t)^2.
\end{align*}

Collecting all the estimates, we obtain
\begin{align}
 & \frac d {dt} \Bigl( \frac 12 \| U^{(n)} (t) \|_{L_x^2}^2 + \frac 12 \| ( U^{(n)} (t) )^\prime \|_{L_x^2}^2 \Bigr) \notag \\
= & - \frac d {dt} \Bigl( \frac 12 \chi_n(t) \int_{\R^2} u^{(n-2)} | \nabla U^{(n)} |^2 dx \Bigr) \notag \\
& \qquad - \frac d {dt} \Bigl( \chi_n(t)  \sum_{1\le i,j\le 2} \int_{\R^2} \partial_{x_j} \vv_i^{(n-2)} V_j^{(n)} \partial_{x_i} U^{(n)} dx \notag \\
&\qquad + \chi_n(t)\sum_{1\le i,j\le 2} \int_{\R^2} \vv_i^{(n-2)} \partial_{x_j} V_j^{(n)} \partial_{x_i} U^{(n)} dx \Bigr) \notag \\
& \qquad - \chi_n(t) \sum_{1\le i,j\le 2} \int_{\R^2} \partial_t V_j^{(n)} \partial_{x_j} \Bigl( \partial_{x_i} U^{(n)} \cdot \vv_i^{(n-2)} \Bigr) dx
\notag \\
& \qquad +\chi_n(t) \int_{\R^2} \partial_t \Bigl( u^{(n-2)} \nabla U^{(n)} \Bigr) \cdot \partial_t V^{(n)} dx \notag \\
& \quad + \tilde E_2, \label{601a}
 \end{align}
 where
 \begin{align*}
 |\tilde E_2| &\lsm \frac {\epsilon} {\langle t \rangle } \Bigl(w_n(t)^2+ w_n(t) \cdot w_{n-1}(t) \Bigr) \\
 &\quad + |\chi_n(t)-\chi_{n-1}(t)| \cdot \frac {\epsilon^2} {\langle t \rangle^2} w_n(t).
 \end{align*}

 Next we deal with the estimate of $V^{(n)}$. This time we need to take the inner product of both sides of \eqref{Nov12b} with
 $\partial_t V^{(n)}$ integrate by parts. Proceeding in a similar manner as in the calculation of $U^{(n)}$ and after a long and tedious
 calculation, we obtain
 \begin{align}
 & \frac d {dt} \Bigl( \frac 12 \| V^{(n)} (t) \|_{L_x^2}^2 + \frac 12 \| ( V^{(n)} (t) )^\prime \|_{L_x^2}^2 \Bigr) \notag \\
= &  \chi_n(t) \sum_{1\le i,j\le 2} \int_{\R^2} \partial_t V_j^{(n)} \partial_{x_j} \Bigl( \partial_{x_i} U^{(n)} \cdot \vv_i^{(n-2)} \Bigr) dx
\notag \\
& \qquad -\chi_n(t) \int_{\R^2} \partial_t \Bigl( u^{(n-2)} \nabla U^{(n)} \Bigr) \cdot \partial_t V^{(n)} dx \notag \\
& \quad + \tilde E_3, \label{601b}
 \end{align}
 where
 \begin{align*}
 |\tilde E_3| &\lsm \frac {\epsilon} {\langle t \rangle } \Bigl(w_n(t)^2+ w_n(t) \cdot w_{n-1}(t) \Bigr) \\
 &\quad + |\chi_n(t)-\chi_{n-1}(t)| \cdot \frac {\epsilon^2} {\langle t \rangle^2} w_n(t).
 \end{align*}

 Adding together \eqref{601a}, \eqref{601b} and rearranging terms, we obtain
 \begin{align}
 & \frac d {dt} \Bigl( \frac 12 \| U^{(n)} (t) \|_{L_x^2}^2 + \frac 12 \| ( U^{(n)} (t) )^\prime \|_{L_x^2}^2  \notag \\
&\qquad + \frac 12 \| V^{(n)} (t) \|_{L_x^2}^2 + \frac 12 \| ( V^{(n)} (t) )^\prime \|_{L_x^2}^2 \notag \\
 &\qquad + \frac 12 \chi_n(t) \int_{\R^2} u^{(n-2)} | \nabla U^{(n)} |^2 dx  \notag \\
&\qquad + \chi_n(t)  \sum_{1\le i,j\le 2} \int_{\R^2} \partial_{x_j} \vv_i^{(n-2)} V_j^{(n)} \partial_{x_i} U^{(n)} dx \notag \\
&\qquad + \chi_n(t)\sum_{1\le i,j\le 2} \int_{\R^2} \vv_i^{(n-2)} \partial_{x_j} V_j^{(n)} \partial_{x_i} U^{(n)} dx \Bigr) \notag \\
= & \tilde E_4, \label{619a}
\end{align}
where
 \begin{align}
 |\tilde E_4| &\lsm \frac {\epsilon} {\langle t \rangle } \Bigl(w_n(t)^2+ w_n(t) \cdot w_{n-1}(t) \Bigr) \notag \\
 &\quad + |\chi_n(t)-\chi_{n-1}(t)| \cdot \frac {\epsilon^2} {\langle t \rangle^2} w_n(t). \label{619b}
 \end{align}

Now denote
\begin{align}
W_n(t) &=
 \frac 12 \| U^{(n)} (t) \|_{L_x^2}^2 + \frac 12 \| ( U^{(n)} (t) )^\prime \|_{L_x^2}^2  \notag \\
&\qquad + \frac 12 \| V^{(n)} (t) \|_{L_x^2}^2 + \frac 12 \| ( V^{(n)} (t) )^\prime \|_{L_x^2}^2 \notag \\
 &\qquad + \frac 12 \chi_n(t) \int_{\R^2} u^{(n-2)} | \nabla U^{(n)} |^2 dx  \label{616a} \\
&\qquad + \chi_n(t)  \sum_{1\le i,j\le 2} \int_{\R^2} \partial_{x_j} \vv_i^{(n-2)} V_j^{(n)} \partial_{x_i} U^{(n)} dx \label{616b} \\
&\qquad + \chi_n(t)\sum_{1\le i,j\le 2} \int_{\R^2} \vv_i^{(n-2)} \partial_{x_j} V_j^{(n)} \partial_{x_i} U^{(n)} dx. \label{616c}
\end{align}

By using \eqref{e127}, \eqref{e100a}--\eqref{e100c}, it is clear that for $\epsilon$ sufficiently small (depending on an absolute constant), we
have
\begin{align*}
& | \eqref{616a} | + | \eqref{616b} | + | \eqref{616c} | \\
\le & \frac 1 {10} \| U^{(n)} (t) \|_{L_x^2}^2 + \frac 12 \| ( U^{(n)} (t) )^\prime \|_{L_x^2}^2  \notag \\
&\qquad + \frac 1 {10} \| V^{(n)} (t) \|_{L_x^2}^2 + \frac 12 \| ( V^{(n)} (t) )^\prime \|_{L_x^2}^2. \notag \\
\end{align*}

Therefore for some absolute constant $C>0$, we have
\begin{align*}
\frac 1 C W_n(t) \le w_n(t) \le C W_n(t).
\end{align*}

The inequalities \eqref{619a} and \eqref{619b} now give us
\begin{align*}
\left|\frac d {dt} W_n(t) \right| &\lsm \frac {\epsilon} {\langle t \rangle} (W_n(t) +W_{n-1}(t) ) \\
& \qquad + |\chi_n(t)-\chi_{n-1}(t)|\cdot \frac {\epsilon^2} {\langle t \rangle^2}, \quad \forall\, t\ge 0,
\end{align*}
where we have canceled a factor of $W_n(t)$ on both sides of the inequality.

We then inductively assume
\begin{align*}
W_{n}(t) \le \frac {\delta_n} { \langle t \rangle^{\frac 12}}.
\end{align*}
Therefore
\begin{align*}
\left| \frac d {dt} W_{n+1}(t) \right| &\lsm   {\delta_n} \cdot \frac {\epsilon}{\langle t \rangle^{\frac 32}}
 + W_{n+1}(t) \cdot \frac {\epsilon} {\langle t \rangle } \\
&\qquad + \frac {\epsilon^2 } {\langle t \rangle^2 } \cdot | \chi(\frac t {2^n}) - \chi ( \frac t {2^{n-1}}) | \\
& \lsm \frac {\delta_n \cdot \epsilon + \epsilon^2 \cdot 2^{-\frac n2}} {\langle t \rangle^{\frac 32}}
+W_{n+1}(t) \cdot \frac {\epsilon}{\langle t \rangle},
\end{align*}
where in the last inequality we have used the fact that the function $|\chi (\frac t {2^n}) - \chi(\frac t {2^{n-1}})|$ is
localized to the region $t\sim 2^n$.
By \eqref{e127}, we have the a priori estimate
\begin{align*}
W_{n+1}(t) \lsm \frac {\epsilon} {\langle t \rangle}.
\end{align*}
We need the following Gronwall lemma
\begin{lem}[Infinite time Gronwall, version 2] \label{gron_2}
Let B(t) be a nonnegative smooth function on $[1,\infty)$ satisfying
\begin{align}
\left|\frac d {dt} B(t) \right| \le \frac \epsilon{t} B(t) + \frac {\delta} {t^{1+\sigma}}
\end{align}
and
\begin{align}
B(t) \le \frac {K}{t^\sigma}, \quad \forall\, t \ge 1.
\end{align}
Here $0<\epsilon<1$, $\sigma>0$, $\delta>0$, $K>0$ are constants. Then for $\epsilon<\sigma$.
\begin{align}
B(t) \le \frac {\delta}{\sigma-\epsilon} \cdot  \frac 1 {t^\sigma}, \quad\, \forall\, t\ge 1,
\end{align}
where $C_1(\sigma)$ is a constant depending only on $\sigma$.
\end{lem}
\begin{proof}
The proof of Lemma \ref{gron_2} is similar to that of Lemma \ref{gron_1}. We omit the
details here.
\end{proof}
By using Lemma \ref{gron_2} and taking $\epsilon$ sufficiently small, we get
\begin{align*}
W_{n+1}(t) \le \frac {\delta_{n+1}} {\langle t \rangle^{\frac 12}},
\end{align*}
where $\delta_{n+1} = \theta \cdot ( \delta_n + 2^{-\frac n2} )$, where $0<\theta<1$ is some constant.
It is obvious that $\delta_n \le const\cdot \alpha^n$ for some $0<\alpha<1$. Hence the sequence
$( \langle t \rangle^{\frac 12} \bar u^{(n+1)}(t), \langle t \rangle^{\frac 12} \bar \vv^{(n+1)}(t) )$
is Cauchy and converges to a nontrivial limiting function in $H^1$.
By interpolation inequalities $(\bar u^{(n+1)}(t), \bar \vv^{(n+1)}(t))$ converges strongly in $H^{m-1}$ for every $t$. For the highest norm
$H^{m}$, upon passing to a subsequence if necessary,  we have $(\bar u^{(n+1)}(t), \bar \vv^{(n+1)}(t))$ tends to the limiting
solution $(\bar u(t), \bar \vv(t))$ weakly for each $t$.
Hence
\begin{align*}
&\| \bar u(t) \|_{H^m} + \| \bar \vv(t) \|_{H^m} + \|(\bar u(t))^\prime \|_{H^m} \\
&\qquad + \| (\bar \vv(t) )^\prime \|_{H^m} \lsm \frac {\epsilon}{\langle t \rangle}.
\end{align*}
We have obtained the desired classical solution.

\end{document}